\documentclass{amsart}
\usepackage[utf8x]{inputenc}
\usepackage{amsmath, amsthm, amssymb}
\usepackage[usenames,dvipsnames,svgnames,table]{xcolor}
\usepackage[margin=1.3in]{geometry}

\newtheorem{theorem}{Theorem}[section]
\newtheorem{proposition}[theorem]{Proposition}
\newtheorem{lemma}[theorem]{Lemma}

\newcommand{\R}{\mathbb R}

\newcommand{\T}{\mathbb T}
\newcommand{\eps}{\varepsilon}

\newcommand{\dd}{\, \mathrm{d}}

\newcommand{\tr}{\mbox{tr}}

\newcommand{\vv}{\langle v\rangle}
\newcommand{\vvo}{\langle v_0 \rangle}

\numberwithin{equation}{section}

\title[Smooth solutions to the Landau equation with hard potentials]{Existence of smooth solutions to the Landau equation with hard potentials and irregular initial data}

\author{Stanley Snelson}
\address{Department of Mathematics and Systems Engineering, Florida Institute of Technology, Melbourne, FL 32901}
\email{ssnelson@fit.edu}

\author{Shelly Ann Taylor}
\address{Department of Mathematics and Systems Engineering, Florida Institute of Technology, Melbourne, FL 32901}
\email{staylor2009@my.fit.edu}

\thanks{Both authors were partially supported by NSF grant DMS-2213407.}
\begin{document}

\maketitle

\begin{abstract}
This paper addresses large-data local existence and uniqueness of classical solutions to the inhomogeneous Landau equation in the hard potentials  case (including Maxwell molecules). Solutions have previously been constructed by Chaturvedi [SIAM J. Math. Anal., 55(5), 5345--5385, 2023] for initial data in an exponentially-weighted $H^{10}$ space, but it is not a priori clear whether these solutions have more regularity than the initial data. We improve Chaturvedi's existence result in two ways: our solutions are $C^\infty$ for positive times, and we allow initial data in a sub-exponentially-weighted $L^\infty$ space, at the cost of requiring a mild positivity condition at time zero. 

To prove uniqueness, we require stronger assumptions on the initial data: H\"older continuity and the absence of vacuum regions. These are the same assumptions that are required for uniqueness in prior work on the soft potentials case.

Along the way to proving existence and uniqueness, we establish some useful results that were previously only known in the case of soft potentials, including spreading of positivity and propagation of H\"older continuity. Many of the proof strategies from the soft potentials case do not apply here because of the more severe loss of velocity moments.
\end{abstract}

\section{Introduction}

We are interested in the Landau equation, a kinetic integro-differential model from plasma physics. 
At time $t\geq 0$, location $x\in \R^3$, and velocity $v\in \R^3$, the equation for the particle density $f(t,x,v)\geq 0$ reads
\begin{equation}\label{e:main}
\partial_t f + v\cdot \nabla_x f = Q(f,f),
\end{equation}
where $Q(f,g)$ is Landau's bilinear collision operator, which acts only in the velocity variable, and is defined by
\[
Q(f,g) = \nabla_v\cdot\left(\int_{\R^3} a(v-w) [f(w)\nabla_v g(v) - f(v)\nabla_w g(w)]\dd w\right),
\]
for any functions $f,g:\R^3\to \R$, and the matrix $a(z)$ is defined by
\[
a(z) = a_\gamma \left( I - \frac{z\otimes z}{|z|^2}\right) |z|^{\gamma+2},
\]
where, in general, $\gamma \in [-3,1]$, $a_\gamma>0$ is a constant depending on $\gamma$, and $I$ is the $3\times 3$ identity matrix. In this paper, we consider $\gamma \in [0,1]$, known as {\it hard potentials} when $\gamma \in (0,1]$ and {\it Maxwell molecules} when $\gamma =0$.

Our goal is to prove existence of a classical solution on some time interval $[0,T]$, given some ``large'' (i.e. not necessarily close to equilibrium) initial data $f_0$. Other aspects of the well-posedness question (global solutions near equilibrium, renormalized solutions, space homogeneous solutions, etc.) will be briefly surveyed in Section \ref{s:related} below.

Local existence for large initial data began with \cite{HeYang}, which addressed the case $\gamma = -3$, and \cite{Henderson-Snelson-Tarfulea2019}, which addressed $\gamma \in [-3,0)$. Most recently, the hard potentials case was addressed by Chaturvedi in \cite{Chaturvedi2019}. This is the most difficult case for local existence because the growth of $a(z)$ for large $z$ leads to a loss of velocity moments in various estimates of the collision operator. In particular, the analysis in \cite{Chaturvedi2019}, which involves a heirarchy of weighted Sobolev norms, is noticeably more intricate than the existence proof for the soft potentials case in  \cite{Henderson-Snelson-Tarfulea2019}.

Regarding the allowable spaces of initial data, all three of the mentioned works \cite{HeYang, Henderson-Snelson-Tarfulea2019, Chaturvedi2019} worked with initial data in Sobolev spaces of high degree (at least 4) with either exponential or high-degree polynomial decay in velocity. The next step in the project of large-data well-posedness is to enlarge the allowable space of initial data $f_0$ to include functions with no, or minimal, regularity hypotheses, while still recovering regularity of $f$ for positive times. (This is a natural goal because the Landau equation is known to have a hypoelliptic smoothing effect \cite{Golse-Imbert-Mouhot-Vasseur2017, Henderson-Snelson2020}.)
This was accomplished for the case $\gamma\in [-3,0)$ in \cite{Tarfulea2020}, which took initial data $f_0$ with $(1+|v|^5) f_0 \in L^\infty$. Again, the proof does not extend naturally to hard potentials because of velocity moment loss in several steps of the argument.

The only prior local existence result for hard potentials is still \cite{Chaturvedi2019}, which does not imply any smoothing for the solution. This leaves open two questions: 

\begin{enumerate}

\item Do the solutions constructed in \cite{Chaturvedi2019} regularize? There is an {\it a priori} smoothing theorem for the hard potentials case \cite{Snelson2020}, but it is not straightforward to apply this result to the solutions contstructed in \cite{Chaturvedi2019} because \cite{Snelson2020} assumes a uniform lower bound on the mass density for positive times.

\item Can the allowable space of initial data be enlarged beyond the space used in \cite{Chaturvedi2019}, which is essentially $\{ f_0 : e^{\rho|v|} f_0 \in H^{10}(\R^6)\}$?

\end{enumerate}

The current article answers these questions affirmatively by constructing a solution to the Landau equation \eqref{e:main} with $\gamma \in [0,1]$, given initial data $f_0$ with no pointwise regularity assumption (besides measurability) and sub-exponential decay in velocity. 
%
We also need to assume $f_0$ is uniformly positive in some small ball in $(x,v)$ space. Therefore, our results are comparable to the results derived in \cite{Tarfulea2020} for the soft potentials case, except that we assume sub-exponential decay in $v$ instead of polynomial. As in \cite{Tarfulea2020}, we require additional hypotheses on $f_0$ (H\"older continuity and a stronger lower bound assumption) to prove uniqueness of solutions.

Along the way to proving existence and uniqueness, we need to extend the result from \cite{Henderson-Snelson-Tarfulea2019} on instantaneous spreading of positivity (i.e. filling of vacuum) to the hard potentials case. This is needed because the regularizing effect of the equation relies on positive lower bounds for the solution. There are nontrivial extra difficulties in proving spreading of positivity that are unique to the hard potentials case, as we briefly discuss in Section \ref{s:proof} below.

One of the technical insights of this study is the use of sub-exponential weight functions of the form $\phi(t,v)  = e^{(\rho - \sigma t)\vv^\beta}$, for some $\beta\in [\gamma,1]$ and $\rho,\sigma >0$. These functions are useful both as barriers in comparison principle arguments,  and as weights in energy estimates. 
Roughly speaking, these sub-exponential functions play a similar role in the current study that inverse-polynomial functions $\vv^{-q}$ play in the soft potentials case (see \cite{Tarfulea2020}). This insight could potentially be applied to the non-cutoff Boltzmann equation, for example in extending the result of \cite{HST2022boltzmann} on existence of classical solutions with irregular data, to the case $\gamma\geq 0$. 

\subsection{Main results}

In both of our main results, we assume that our initial data is periodic in the $x$ variable. This assumption allows us to ignore some technicalities and focus on the key ideas, particularly the difficulties that are specific to the hard potentials regime. The periodicity condition could be removed using techniques similar to those seen in \cite{Tarfulea2020, HST2022boltzmann}.

Recall the notation $\vv = \sqrt{1+|v|^2}$. Our first result is on local existence with large initial data:

\begin{theorem}\label{t:existence}
Let $\gamma \in [0,1]$, and let $f_0 :\R^3_x\times \R^3_v \to [0,\infty)$ be periodic in the $x$ variable and satisfy
\[
\|e^{\rho \vv^\beta} f_0\|_{L^\infty(\R^6)} \leq K_0,
\]
for some $\rho, K_0>0$ and $\beta$ satisfying
\[
\begin{cases}
\beta \in [\gamma, 1], &\text{if $\gamma \in (0,1],$}\\
\beta = 1, &\text{if $\gamma = 0$.}
\end{cases}
\]
Furthermore, if $\gamma \in [0,1)$, assume there exist $(x_m,v_m)\in \R^6$ and $\delta, r>0$ such that
\[
f_0(x,v) \geq \delta, \quad |x-x_m|< r, |v-v_m|< r.
\]
If $\gamma = 1$, assume there exist $\delta, r, R>0$ such that for every $x_m\in \R^3$, there exists a $v_m\in B_R(0)$ with 
\[
f_0(x,v) \geq \delta, \quad |x-x_m|< r, |v-v_m|< r.
\]

Then there exist $T,\sigma >0$ depending on $\gamma$, $\beta$, and $K_0$ (but not on $\delta$, $r$, or $R$) and a classical solution $f$ to the Landau equation \eqref{e:main} on $[0,T]\times\R^3\times\R^3$ such that $e^{(\rho - \sigma t)\vv^\beta} f(t) \in L^\infty(\R^6)$ for each $t\in [0,T]$. This solution is infinitely differentiable in $(0,T]\times\R^3\times\R^3$, with any partial derivative in $(t,x,v)$ variables bounded uniformly on any compact subset of $(0,T]\times\R^3\times\R^3$. 

The solution $f$ agrees with the initial data in the following sense: for any test function $\phi \in C^{1}_{t,x} C^2_v$ with compact support in $[0,T)\times\R^3\times\R^3$,
\[
\int_{\R^6} f_0(x,v) \phi(0,x,v) \dd v\dd x = \int_0^T \int_{\R^6} [f(\partial_t + v\cdot\nabla_x)\phi + Q(f,f) \phi]\dd v\dd x \dd t.
\]
%
%
%
%
%
\end{theorem}

Several comments on the statement of Theorem \ref{t:existence} are in order:
\begin{itemize}

\item The stronger decay assumption when $\gamma = 0$ is a technical condition that could likely be removed in a more detailed study of the specific case $\gamma = 0$. The analysis in this paper is designed with the case $\gamma\in (0,1]$ primarily in mind, but happens to also give a solution for $\gamma = 0$ under the stated hypotheses.

\item The extra lower bound condition when $\gamma = 1$ is needed because our proof of spreading of positivity breaks down in this case. The same issue occurs for the non-cutoff Boltzmann equation, which is why the lower bound theorem in \cite{HST2020lowerbounds} also requires $\gamma<1$.

\item If the initial data is continuous, then it can be shown that $f(t,x,v) \to f_0(x,v)$ pointwise as $t\to 0$, as expected. The proof is identical to \cite[Proposition 3.1]{Tarfulea2020}, so we omit it.

\item If the spatial domain were not periodic, one would need to assume that lower bounds for $f_0$ are ``well-distributed'' in the sense of \cite{Tarfulea2020}, i.e. that no $x$ location is too far away from a location where $f$ satisfies positive lower bounds.

\end{itemize}

%

Next, we state our main uniqueness result. As in the corresponding study of the Landau equation with soft potentials \cite{Tarfulea2020} and the related work \cite{HST2022boltzmann} on the Boltzmann equation, stronger assumptions are required to prove uniqueness than existence. The kinetic H\"older space $C^\alpha_{k,x,v}$ is defined in Section \ref{s:kinholder} below.

\begin{theorem}\label{t:uniqueness}
Let $f_0:\R^3_x\times\R^3_v\to [0,\infty)$ be periodic in the $x$ variable and satisfy
\[
\|e^{\rho_0 \vv^\beta} f_0\|_{C^\alpha_{k,x,v}(\T^3\times\R^3)} \leq K_0,
\]
for some $\rho_0, K_0>0$, $\alpha\in (0,1)$, and $\beta$ as in Theorem \ref{t:existence}. Furthermore, assume there are $\delta, r, R>0$ such that for every $x_m \in \R^3$, there exists $v_m\in B_R(0)$ with 
\[
f_0(x,v) \geq \delta, \quad |x-x_m|< r, |v-v_m|< r.
\]
Let $f:[0,T]\times \T^3\times\R^3 \to [0,\infty)$ be the solution to the Landau equation \eqref{e:main} guaranteed by Theorem \ref{t:existence}. 

Then there exists $T_U\in (0,T]$, depending on $K_0$, $\rho_0$, and $\alpha$, such that for any classical solution $g\geq 0$ to \eqref{e:main} on $[0,T_U]\times\T^3\times\R^3$ with $g(0,x,v) = f_0(x,v)$ and
\[
\sup_{t,x} \int_{\R^3} (1 + |v|^{\gamma+2}) g(t,x,v) \dd v < \infty,
\]
there must hold $f=g$. 
\end{theorem}


The uniqueness or non-uniqueness of the solutions constructed in Theorem \ref{t:existence}, when the extra hypotheses of Theorem \ref{t:uniqueness} are not satisfied, remains an open question. 

\subsection{Related work}\label{s:related}

Prior work on the inhomogeneous Landau equation with large initial data has mainly focused on short-time existence \cite{HeYang, Henderson-Snelson-Tarfulea2019, Chaturvedi2019, Tarfulea2020} as well as conditional regularity and continuation criteria \cite{Golse-Imbert-Mouhot-Vasseur2017, Henderson-Snelson2020, henderson2021schauder, snelson2023landau}. Global existence in the large-data case is a challenging open problem that may be out of reach with current techniques. Short-time existence and conditional regularity studies are partially motivated by shedding light on this problem (ruling out certain types of singularities).

Existence theory for the Landau equation is studied in several other regimes, which we briefly survey now.

For global-in-time solutions close to a Maxwellian equilibrium state, see  \cite{guo2002periodic, carrapatoso2016cauchy, CM2017verysoft, kim2020landau, duan2019mild, guo2020landau, golding2023global} and the references therein. See \cite{luk2019vacuum, chaturvedi2020vacuum} for global solutions close to the vacuum state $f\equiv 0$.

In the space homogeneous case, where $f$ is independent of $x$, the Landau equation is known to be globally well-posed, even for large initial data. This was shown in \cite{Desvillettes-Villani2000} for the case $\gamma\in (0,1]$, \cite{villani1998landau} for $\gamma = 0$, \cite{Wu2013,alexandre2015apriori} for $\gamma\in [-2,0)$, and finally \cite{guillen2023landau} for $\gamma \in [-3, -2)$. See \cite{fournier2010uniqueness,  silvestre2015landau, gualdani2017landau, CDE2017landau, carillo2020gradient, GGIV2022landau, golse2022local, desvillettes2023mono, ABDL, chen2023landau, golding2024, gualdani2024blowdown} and the references therein for other results on the existence and regularity of the homogeneous Landau equation. 

A suitable notion of weak solution, called renormalized solutions with defect measure, was constructed in \cite{villani-ren} for very general initial data. The regularity and uniqueness of these solutions is not understood.

\subsection{Difficulties and proof strategy}\label{s:proof}

As mentioned above, the results in this paper are in a similar spirit to \cite{Tarfulea2020}, which considered the case of soft potentials ($\gamma \in [-3,0)$) and used norms with polynomial velocity weights. In general, the current study has to deal with the main difficulties encountered in \cite{Tarfulea2020}  in addition to the new issues brought about by the stronger loss of moments when $\gamma >0$.

Let us now discuss the proof strategies for three main areas of this work.

\subsubsection{Spreading of positivity}

A key tool in \cite{Tarfulea2020} is the positivity-spreading result of \cite{Henderson-Snelson-Tarfulea2019}, which was proven via a probabilistic argument that requires $\gamma < 0$ in an apparently essential way. In the current work (see Theorem \ref{t:lowerbounds}), we extend positivity-spreading to $\gamma\geq 0$ via a deterministic barrier argument inspired by a similar argument from the study of the Boltzmann equation \cite{HST2020lowerbounds}. The basic steps of this argument are as follows: (i)  propagate local lower bounds forward in time, (ii) spread lower bounds to high velocities, (iii) spread lower bounds to any desired point $(t,x)$, using the fact that $f$ has positive lower bounds at the necessary velocity, (iv) spread to large velocities again in a neighborhood of $(t,x)$. 

In the step of spreading to large $|v|$, one cannot adapt the proof given for the Boltzmann equation in \cite{HST2020lowerbounds} (which borrowed some techniques from \cite{Imbert2020lowerbounds}) because the Landau collision operator has worse coercivity properties than the Boltzmann collision operator, in some sense. This is related to the fact that $Q(f,g)$ depends only locally on $g$, unlike in the Boltzmann case where one can prove coercive lower bounds at a point $v_0\in \R^3$ using positivity of $g$ in regions far from $v_0$. 

Instead, we use a barrier argument to spread in velocity, but because of the usual moment loss when $\gamma>0$, one can only close this argument by localizing in $v$, with the end result that we cannot spread lower bounds to arbitrarily large velocity. Despite this, the lower bounds we obtain are sufficient to access the regularizing properties of the collision term $Q(f,f)$ that are needed in our proofs of existence and uniqueness.

If one assumes that $f_0$ satisfies lower bounds that are uniform in $x$ (such as the hypotheses in Theorem \ref{t:uniqueness}), then $f$ will satisfy global lower bounds in $v$ with the optimal Gaussian asymptotics (see \cite{Snelson2020}). It is not clear whether global lower bounds in $v$ are true without these additional assumptions. This is different from what is seen in the Boltzmann case \cite{HST2020lowerbounds} and the soft potentials case of Landau \cite{Henderson-Snelson-Tarfulea2019}.

\subsubsection{Existence}

The key lemma for local existence in the soft potentials case \cite{Tarfulea2020} is an a priori estimate in the polynomially weighted space $L^\infty_q(\R^6)$ of the form $\|f(t)\|_{L^\infty_q(\R^6)} \leq C \|f_0\|_{L^\infty_q(\R^6)}$ for $t$ small enough. This estimate relies on the fact that $h(t,v) = e^{\beta t} \vv^{-q}$ is a supersolution of the linear Landau equation for suitable $\beta$ and $q$, but this fact is false when $\gamma >0$, because the collision term grows too fast in velocity. 

In this work, the polynomially-decaying function $h$ is replaced with sub-exponentially decaying supersolutions $\phi(t,v) = e^{(\rho-\sigma t) \vv^\beta}$, for some $\beta\in [\gamma,1]$. The benefit of $\phi$ is that $\partial_t \phi$ produces a term $-\sigma \vv^\beta \phi$, which has the right sign to absorb the extra moments produced by $\bar c^f$ when $\gamma >0$. This provides an estimate of the form $\|e^{(\rho-\sigma t) \vv^\beta} f(t)\|_{L^\infty(\R^6)} \leq C \|e^{\rho\vv^\beta} f_0\|_{L^\infty(\R^6)}$ for small enough $t$. 

Once this estimate is available, existence follows by an approximation argument similar to \cite{Tarfulea2020, HST2022boltzmann}, using the earlier existence result of Chaturvedi \cite{Chaturvedi2019} and regularity estimates for the Landau equation (see Appendix \ref{s:a}).

\subsubsection{Uniqueness}

There is a fundamental difficulty, seen already in \cite{Tarfulea2020,HST2022boltzmann}, with proving uniqueness in the setting of low-regularity initial data. Namely, to control the difference between two solutions $f$ and $g$, one needs some velocity regularity for one of the solutions (say $f$). Naively applying regularity estimates yields a constant that blows up too fast as $t\to 0$ to be useful. Therefore, one must take initial data that is (in addition to satisfying the assumptions of our existence theorem) H\"older continuous and free of vacuum regions. By studying the evolution of a finite difference of the solution, one can propagate the H\"older modulus forward to positive times, as in \cite{Tarfulea2020}. This provides enough regularity to prove uniqueness.

In the hard potentials case, the same overall strategy works, but implementing the details requires, as usual, controlling extra velocity moments. As in the proof of existence, sub-exponential weights like $e^{(\rho - \sigma t)\vv^\beta}$ are needed because of the good term produced by the time derivative falling on the weight.

\subsection{Notation} 

For $f,g:\R^3\to \R$, it is well-known that Landau's collision operator $Q(f,g)$ can be written in either divergence form
\[
Q(f,g) = \nabla_v\cdot(\bar a^f \nabla_v g) + \bar b^f \cdot \nabla_v g + \bar c^f g,
\]
or non-divergence form
\[
Q(f,g) = \tr(\bar a^f D_v^2 g) + \bar c^f g,
\]
where the nonlocal coefficients are defined by
\begin{align}
\bar a^f(v) &:= a_{\gamma}\int_{\R^3} \left( I - \frac w {|w|} \otimes \frac w {|w|}\right) |w|^{\gamma + 2} f(v-w) \dd w,\label{e:a}\\
\bar b^f(v) &:= b_{\gamma}\int_{\R^3} |w|^\gamma w f(v-w)\dd w,\label{e:b}\\
\bar c^f(v) &:= c_{\gamma}\int_{\R^3} |w|^\gamma f(v-w)\dd w, \label{e:c}
\end{align}
where $I$ is the $3\times 3$ identity matrix and $a_{\gamma}$, $b_{\gamma}$, $c_{\gamma}$ are constants depending on $\gamma$. When $f$ is a function of $(t,x,v)$, then these coefficients naturally also depend on all three variables: $t$, $x$, and $v$.

We denote polynomially-weighted $L^p$ norms with the notation
\[
\|f\|_{L^p_q(\Omega)} = \|\vv^q f\|_{L^p(\Omega)},
\]
for any $p\in [1,\infty]$ and $q\geq 0$. Here, $\Omega$ could be a subset of $\R^3_v$ or $\R^3_x\times\R^3_v$ or $[0,T]\times\R^3_x\times\R^3_v$.

We often use the notation $z = (t,x,v)$ to denote a point in $\R^7$. 

Our solutions are periodic in $x$. The period itself does not play any quantitative role in our proofs, so we assume the spatial domain is the standard torus $\T^3$. Equivalently, we can consider the domain to be $\R^3$, with our solutions extended by periodicity. We use these two notations for the spatial domain interchangeably.

Throughout this paper, all constants may depend on the parameter $\gamma$, even when not explicitly noted.
\subsection{Outline of the paper}
Section \ref{s:prelim} presents some results from the literature that we need, as well as some preliminary lemmas. Section \ref{s:decay} proves decay estimates for large velocities, Section \ref{s:lower} establishes the spreading of positivity, and Section \ref{s:existence} completes the proof of existence. Section \ref{s:holder} proves that H\"older continuity is propagated forward in time, and Section \ref{s:uniqueness} proves uniqueness. Appendix \ref{s:a} provides the regularity analysis that we need in our work.

\section{Preliminaries and known results}\label{s:prelim}

\subsection{Existence for regular initial data}

The following existence result was proven by Chaturvedi \cite{Chaturvedi2019}. We state a simplified version of his theorem with less sharp hypotheses, that is sufficient for our purposes:

\begin{theorem}\label{t:chaturvedi}
Let $M_0 > 0,\gamma \in [0,1],d_0 > 0,f_0$ be such that
\[
{\sum}_{{|\alpha|}+{|\beta|}\leq 10} || \partial^{\alpha}_x \partial^{\beta}_v (e^{d_0 \langle v \rangle} f_0)||^2_{L^2(\R^6)} \leq M_0.
\]
Then for some $T > 0$, depending on $\gamma$, $d_0$ and $M_0$, there is a non-negative solution $f$ to
the Landau equation on $[0,T]\times\R^6$ with $f(0, x, v) = f_0(x, v).$

Moreover, $e^{(d_0 - \kappa t) \langle v\rangle} f\in C([0,T],H^{10}_{x,v}(\mathbb R^6))$, for some $\kappa>0$ depending on $\gamma$, $d_0$, and $M_0$. 
\end{theorem}

Below, we will construct a solution using an approximation procedure involving the solutions provided by Theorem \ref{t:chaturvedi}. 

\subsection{Coefficient bounds}

First, we have upper bounds on the coefficients in terms of $L^1$ moments of the solution, which are by now standard in the literature:
\begin{lemma} \cite[Lemma 2.1]{Snelson2020} \label{l:abc}
Let $f$ satisfy 
\[
\int_{\R^3}(1 + |v|^{\gamma+2}) f(t,x,v) \dd v \leq K_0, \quad \text{for all }  t\in [0, T], x\in \R^3.
\]
Then there exist constants $C_1$, $C_2$ $C_3$, depending only on $\gamma$, such that 
\begin{align*}
 \bar a^f_{ij}(t,x,v) \xi_i \xi_j &\leq C_1 K_0|\xi|^2
  \begin{cases} (1+|v|)^{\gamma+2}, &\xi \perp v,\\
(1+|v|)^{\gamma},  & \xi \parallel v,\end{cases}\\
|\bar b^f(t,x,v)| &\leq C_2 K_0(1+|v|)^{\gamma+1},\\
\bar c^f(t,x,v) &\leq C_3 K_0 (1+|v|)^\gamma.
\end{align*} 
\end{lemma}

Note that the weighted norm $\|f\|_{L^\infty_q(\R^3)} = \|\vv^q f\|_{L^\infty(\R^3)}$ controls the integral quantity $\int_{\R^3} (1+|v|^{\gamma+2}) f(t,x,v) \dd v$ when $q>\gamma+5$. Therefore, the upper bounds of Lemma \ref{l:abc} also hold with $K_0$ replaced by $\|f\|_{L^\infty_q([0,T]\times\R^6)}$, and we sometimes use this alternate statement of the upper bounds without comment.


Furthermore, we have a lower ellipticity estimate for the matrix $\bar a^f$. The proof is the same as \cite[Lemma 4.3]{Henderson-Snelson-Tarfulea2019}.

\begin{lemma}\label{l:coercivity}
If $f:\R^3\to \R$ is nonnegative and satisfies the lower bound
\[
f(v)\geq \delta, \quad  \text{for all $v\in B_r(v_0),$}
\]
for some $r,\delta>0$ and $v_0\in \R^3$, then for all $v\in \R^3$, the matrix $\bar a^f(v)$ defined by \eqref{e:a} satisfies, for $\xi\in \R^3$,
\begin{equation}\label{e:a-lower-bounds}
\bar a^f_{ij}(v) \xi_i\xi_j \geq c_1|\xi|^2 \begin{cases} (1+|v|)^{\gamma}, & \xi \parallel v,\\
 (1+|v|)^{\gamma+2}, &\xi \perp v,
 \end{cases}
\end{equation}
with $c_1>0$ depending only on $\delta$, $r$, and $|v_0|$. 
\end{lemma}

\subsection{Kinetic H\"older norms}\label{s:kinholder}

The regularity of the inhomogeneous Landau equation is most naturally measured with respect to a metric which respects the invariance of kinetic equations with respect to rescalings of the form $(t,x,v)\mapsto (r^2 t, r^3x, rv)$ and Galilean shifts. In more detail, define the kinetic distance
\[
d_k(z,z') = |t-t'|^{1/2} + |x' - x - (t'-t)v|^{1/3} + |v'-v|.
\]
Technically, $d_k$ is not a metric on $\R^7$ because it does not satisfy the triangle inequality and is not symmetric. However, this fact causes no issues in our analysis. For any $\alpha \in (0,1)$ and domain $\Omega\subset \R^7$, we define the kinetic H\"older seminorm
\[
[u]_{C^\alpha_k(\Omega)} = \sup_{z,z'\in \Omega} \frac{|u(z) - u(z')|}{d_k(z,z')^\alpha},
\]
as well as the norm $\|u\|_{C^\alpha_k(\Omega)} = \|u\|_{L^\infty(\Omega)} + [u]_{C^\alpha_k(\Omega)}$, and the H\"older space $C^\alpha_k(\Omega) = \{u: \Omega \to \R, \|u\|_{C^\alpha(\Omega)} < \infty\}$. 

Let us also define the second-order space $C^{2,\alpha}_k(\Omega)$ using the norm
\[
\begin{split}
\|u\|_{C^{2,\alpha}_k(\Omega)} &= \|u\|_{L^\infty(\Omega)} + \|\nabla_v u\|_{L^\infty(\Omega)} + \|D_v^2 u\|_{C^\alpha_k(\Omega)} + \|(\partial_t + v\cdot\nabla_x)u\|_{C^\alpha_k(\Omega)}.
\end{split}
\]
We sometimes apply the norm $\|\cdot\|_{C^{2,\alpha}}$ to functions $u$ where $\partial_t u$ and $\nabla_x u$ may not be defined pointwise. In this case, the differential operator $(\partial_t + v\cdot \nabla_x)$ has been extended by density. 

We also define the kinetic H\"older norm in $(x,v)$ variables for subsets $\Omega\subset\R^3 \times \R^3$:
\[
[u]_{C^{\alpha}_{k,x,v}(\Omega)} = \sup_{(x,v),(x',v') \in \Omega} \frac{|u(x,v) - u(x',v')|}{d_k((0,x,v),(0,x',v'))^\alpha}, \quad \|u\|_{C^\alpha_{k,x,v}(\Omega)} = \|u\|_{L^\infty(\Omega)} + [u]_{C^\alpha_{k,x,v}(\Omega)}.
\]

Next, we recall the standard kinetic cylinders, defined for some point $z_0 \in \R^7$ and radius $r>0$ by
\[
\begin{split}
Q_r(z_0) &= \{ z=(t,x,v) \in \R^7 : t < t_0 \text{ and } d_k(z,z_0) < r\}\\
&= \{(t,x)\in \R^4 : t_0-r^2 < t < t_0 \text{ and } |x-x_0 - (t-t_0)v_0| < r^3\} \times B_r(v_0).
\end{split}
\]
We also use the notations $Q_r = Q_r(0)$ and 
\[
Q_r^{t,x}(z_0) = \{(t,x)\in \R^4 : t_0-r^2 < t < t_0 \text{ and } |x-x_0 - (t-t_0)v_0| < r^3\}.
\]

The following is a standard result about the H\"older seminorm of a product, which we state without proof:
\begin{lemma}\label{l:product}
For any subset $\Omega\subset \R^7$, and any $f,g\in C^\alpha_k(\Omega)$, the following inequality holds:
\[
[fg]_{C^\alpha_k(\Omega)} \leq \|f\|_{L^\infty(\Omega)}[g]_{C^\alpha(\Omega)} + [f]_{C^\alpha(\Omega)}\|g\|_{L^\infty(\Omega)}.
\]
\end{lemma}

\subsection{Regularity estimates}

In our proof of uniqueness, we need to use regularity estimates in H\"older spaces for the Landau equation. These estimates incorporate known estimates for the linear Landau equation \cite{Golse-Imbert-Mouhot-Vasseur2017, Henderson-Snelson2020} while taking into account the dependence of the coefficients $\bar a^f$, $\bar b^f$, $\bar c^f$ on the solution $f$. We also need to understand how fast the constants in these nonlinear estimates degenerate as $t\to 0$.

This has been worked out previously for the soft potentials case in \cite{Tarfulea2020}. In the hard potentials case, this regularity theory follows the same main ideas, but some of the details of the analysis are different, such as the exponents that appear in the statements of some results. Therefore, we state here the main regularity estimate that we need, and provide a full proof in Appendix \ref{s:a}.

\begin{proposition}\label{p:D2f-est}
Let $f$ be a solution the Landau equation on $[0,T]\times\R^6$, and for some $\tau\in (0,T]$, $\rho_0>0$, $\beta \in [\gamma,1]\setminus\{0\}$, and $\alpha\in (0,1)$, assume that
\[
 e^{\rho_0\vv^\beta} f \in C^\alpha_k([0,\tau]\times\R^6).
\]
and 
\begin{equation}
\bar a^f_{ij}(t,x,v)\xi_i \xi_j \geq
 \lambda_0 \begin{cases}
\vv^\gamma, &\xi \perp v,\\
\vv^{\gamma+2}, & \xi \parallel v,
\end{cases}
\quad \text{for all $(t,x,v)\in [0,\tau]\times\R^6$ and $\xi\in \R^3$.}
\end{equation}
Then 
\[
\|e^{\rho \vv^\beta} D_v^2 f\|_{L^\infty([\tau/2,\tau]\times\R^6)} \leq C \left(1 + \tau^{-1 + \frac{\alpha^2}{6-\alpha}}\right)\left(1+\|e^{\rho_0\vv^\beta}  f\|_{C^\alpha_k([0,\tau]\times\R^6)}^{P(\alpha)}\right),
\]
where $\rho = \dfrac \alpha {6-2\alpha} \rho_0$, $P(\alpha)>1$ is an exponent depending only on $\alpha$, and $C>0$ is a constant depending only on $\rho_0$, $\beta$, $\alpha$, $\lambda_0$, and $\|e^{\rho_0\vv^\beta} f\|_{L^\infty([0,\tau]\times\R^6)}$. 
\end{proposition}

Next, we have a proposition that says that $(x,v)$ regularity implies $t$ regularity, for solutions of a class of linear kinetic equations that include the (linear) Landau equation. A similar result was shown in \cite[Proposition A.1]{Tarfulea2020}, under a stronger assumption on the coefficients, namely that the zeroth-order coefficient $c$ is uniformly bounded. To prove the more general form that we state here, one can modify the proof in \cite{Tarfulea2020} in a straightforward way to account for a $c(t,x,v)$ that grows polynomially in $v$. Therefore, we omit the proof.
\begin{proposition}\label{p:xv-to-t}
Suppose that $f:[0,T]\times\R^6 \to [0,\infty)$ is a solution of the linear equation
\[
\partial_t f + v\cdot\nabla_x f = \textup{\mbox{tr}}(a D_v^2 f) + cf,
\]
where the coefficients $a$ and $c$ satisfy
\[
\|\vv^{\gamma+2} a_{ij}\|_{L^\infty([0,T]\times\R^6)} + \|\vv^\gamma c\|_{L^\infty([0,T]\times\R^6)} \leq K_0.
\]
Furthermore, assume $f$ is locally H\"older continuous in $(x,v)$ variables, and that $\vv^\gamma f$ is bounded. Then $f$ is locally H\"older continuous in all three variables, and the estimate
\[
\|f\|_{C^\alpha_k(Q_1(z_0)\cap [0,T]\times\R^6)} \leq C \langle v_0\rangle^{\alpha(1+\gamma/2)+\gamma}\left( \|f|_{L^\infty_\gamma([0,T]\times\R^6)} + \sup_{\substack{0\leq t\leq t_0\\t_0-t\leq 1}} [f(t,\cdot,\cdot)]_{C^\alpha_{k,x,v}(B_2(x_0,v_0))} \right)
\]
holds, where $C>0$ is a constant depending only on $\gamma$ and $K_0$. 
\end{proposition}

\subsection{Sub-exponential functions}

Throughout the paper, we make use of functions of the form  $\phi = e^{\rho \vv^\beta}$ for some $\rho \in \R$ and $\beta>0$. Sometimes, $\rho$ will be replaced by a linear function of $t$. Let us collect a few useful properties: 
\begin{align}
\partial_{v_i} \phi &= \rho \beta \phi \vv^{\beta-2} v_i , \quad i = 1,2,3,\label{e:phi-first}\\
\partial_{v_i v_j} \phi &=  \rho \beta \vv^{\beta-4} \phi \left[ (\beta-2)  v_i v_j + \vv^{2} \delta_{ij} + \rho \beta \vv^{\beta} v_i v_j\right], \quad i,j = 1,2,3,\label{e:phi-second}\\
\bar a^g_{ij} \partial_{v_i v_j} \phi &= \rho \beta \vv^{\beta-4}\phi\left[ \left((\beta - 2)  + \rho \beta \vv^{\beta}\right) \bar a^g_{ij} v_i v_j + \vv^{2} \tr(\bar a^g)\right] \label{e:a-phi},
\end{align}
where $\bar a^g$ is defined by \eqref{e:a} for any function $g$, and the expression in \eqref{e:a-phi} is summed over repeated indices.

We also have two interpolation lemmas with sub-exponential weights, that will be used in our proof of the propagation of a H\"older modulus:

\begin{lemma}\label{l:interp1}
For any $\theta, \alpha\in (0,1]$, any $z_0\in \R^7$ any $\beta \in [0,1]$, any $\rho_1 \geq \rho_0 \geq 0$, and any $g:Q_\theta(z_0) \to \R$ such that the right-hand side is finite, there holds
\begin{equation}\label{e:claim}
[e^{[(\rho_0+\rho_1)/2] \vv^\beta} g]_{C^{\alpha/2}_k(Q_\theta(z_0))} \leq C [e^{\rho_0 \vv^\beta} g]_{C^{\alpha}_k(Q_\theta(z_0))}^{1/2}\|e^{\rho_1  \vv^\beta} g\|_{L^\infty(Q_\theta(z_0))}^{1/2},
\end{equation}
for a constant $C$ depending on $\rho_0$, $\rho_1$, $\beta$, and $\alpha$. 

Furthermore, the same interpolation holds for functions defined on $[0,T]\times\R^6$:
\[
[e^{[(\rho_0+\rho_1)/2] \vv^\beta} g]_{C^{\alpha/2}_k([0,T]\times\R^6)} \leq C [e^{\rho_0 \vv^\beta} g]_{C^{\alpha}_k([0,T]\times\R^6)}^{1/2}\|e^{\rho_1  \vv^\beta} g\|_{L^\infty([0,T]\times\R^6)}^{1/2},
\]
with $C$ as above.
\end{lemma}

\begin{proof}
Define
\[
R  = \left(\frac{ \|e^{\rho_1 \vv^\beta} g\|_{L^\infty(Q_\theta(z_0))} } { e^{[(\rho_1-\rho_0)/2]\langle v_0 \rangle^\beta} [e^{\rho_0 \vv^\beta} g]_{C^{\alpha}_k(Q_\theta(z_0))} } \right)^{2/\alpha}.
\]
Taking distinct $z_1, z_2 \in Q_\theta(z_0)$, there are two cases. If $d_k(z_1,z_2) \geq R$, then
\[
\begin{split}
\frac{|g(x_1,v_1) - g(x_2,v_2)|}{d_k(z_1,z_2)^{\alpha/2}} &\leq C R^{-\alpha/2} e^{-\rho_1 \langle v_0\rangle^\beta} \| e^{\rho_1 \langle v_0\rangle^\beta} g\|_{L^\infty(Q_\theta(z_0))}\\
&= C  e^{-[(\rho_0+\rho_1)/2] \langle v_0 \rangle^\beta} [e^{\rho_0 \vv^\beta} g]_{C^{\alpha}_k(Q_\theta(z_0))}^{1/2} \|e^{\rho_1 \vv^\beta} g\|_{L^\infty(Q_\theta(z_0))}^{1/2}.
\end{split}
\]
If $d_k(z_1,z_2) < R$, then with Lemma \ref{l:product}, we have
\[
\begin{split}
\frac{|g(x_1,v_1) - g(x_2,v_2)|}{d_k(z_1,z_2)^{\alpha/2}} &\leq \frac{|g(x_1,v_1) - g(x_2,v_2)|}{d_k(z_1,z_2)^{\alpha}} d_k(z_1,z_2)^{\alpha/2}\\
&\leq Ce^{-\rho_0 \langle v_0 \rangle^\beta} [e^{\rho_0 \vv^\beta} g]_{C^{\alpha}_k(Q_\theta(z_0))} R^{\alpha/2}\\
&= Ce^{-[(\rho_0+\rho_1)/2] \langle v_0 \rangle^\beta} [e^{\rho_0 \vv^\beta} g]_{C^{\alpha}_k(Q_\theta(z_0))}^{1/2} \|e^{\rho_1 \vv^\beta} g\|_{L^\infty(Q_\theta(z_0))}^{1/2}.
\end{split}
\]
In either case, we see that  $e^{[(\rho_0+\rho_1)/2] \langle v_0 \rangle^\beta}[g]_{C^{\alpha/2}_k(Q_\theta(z_0))}$ is bounded by the right-hand side of \eqref{e:claim}. Conclusion \eqref{e:claim} then follows after applying Lemma \ref{l:product} again.

To prove the interpolation inequality on the whole space, cover $[0,T]\times\R^6$ with a countable union of kinetic cylinders with radius 1 centered at $z_i$, and note that
\[
\|e^{[(\rho_0 + \rho_1)/2]\vv^\beta} g\|_{C^{\alpha/2}_k([0,T]\times\R^6)} \approx \sum_{i=1}^\infty \|e^{[(\rho_0 + \rho_1)/2]\vv^\beta} g\|_{C^{\alpha/2}_k(Q_1(z_i)\cap([0,T]\times\R^6))}.
\]
The inequality then follows from applying \eqref{e:claim} for each $z_i$. 
\end{proof}

\begin{lemma}\label{l:exp-interp}
For $g:\R^6\to \R$ such that the right-hand side is finite, there holds for any $z_0\in \R_+ \times \R^6$ and any $\theta \in (0, \min\{1,\sqrt{t_0/2}\})$,
\[
\|e^{\rho\vv^\beta} D_v^2g\|_{L^\infty(Q_\theta(z_0))} \leq C   [D_v^2 g]_{C^{2\alpha/3}_k(Q_\theta(z_0))}^{1 - \frac {2\alpha} {6-\alpha} }  \|e^{\rho'\vv^\beta} g\|_{C^\alpha_k(Q_\theta(z_0))}^{\frac {2\alpha} {6-\alpha}},
\]
with $\rho' = \rho \left(\dfrac 6 \alpha - 2 \right)$, and $C>0$ a constant depending on $\rho$, $\alpha$, and $\beta$. 
\end{lemma}

\begin{proof}
First, we apply a standard unweighted interpolation between $C^2$, $C^{2,\alpha/3}$, and $C^\alpha$ norms, obtaining
\[
\|e^{\rho\vv^\beta} D_v^2g\|_{L^\infty(Q_\theta(z_0))} \leq [e^{\rho\vv^\beta} D_v^2 g]_{C^{\alpha/3}_k(Q_\theta(z_0))}^{1 - \frac \alpha {6-2\alpha}} [e^{\rho\vv^\beta} g]_{C^\alpha_k(Q_\theta(z_0))}^{\frac \alpha{6-2\alpha}}.
\]
Next, we apply Lemma \ref{l:interp1} to $D_v^2 g$, with $2\alpha/3$ replacing $\alpha$, and with $\rho_0 = 0$ and $\rho_1 = 2\rho$:
\begin{equation}\label{e:intermediate-2}
\begin{split}
\|e^{\rho\vv^\beta} D_v^2g\|_{L^\infty(Q_\theta(z_0))}
 &\leq [D_v^2 g]_{C^{2\alpha/3}_k(Q_\theta(z_0))}^{\frac 1 2 - \frac {\alpha/2} {6-2\alpha}} \|e^{2\rho\vv^\beta} D_v^2 g\|_{L^\infty(Q_\theta(z_0))}^{\frac 1 2 - \frac {\alpha/2}{6-2\alpha}} [e^{\rho\vv^\beta} g]_{C^{\alpha}_k(Q_\theta(z_0))}^{\frac \alpha{6-2\alpha}}\\
&\leq C e^{\rho \langle v_0\rangle^{\beta}(1 -\frac \alpha {6-2\alpha})} [D_v^2 g]_{C^{2\alpha/3}_k(Q_\theta(z_0))}^{\frac 1 2 - \frac {\alpha/2} {6-2\alpha}} \|e^{\rho\vv^\beta} D_v^2 g\|_{L^\infty(Q_\theta(z_0))}^{\frac 1 2 -\frac {\alpha/2} {6-2\alpha}} [e^{\rho\vv^\beta} g]_{C^\alpha_k(Q_\theta(z_0))}^{\frac \alpha{6-2\alpha}}.
\end{split}
\end{equation}
To absorb the exponential factor, we apply Lemma \ref{l:product}:
\[
\begin{split}
&e^{\rho \langle v_0\rangle^{\beta}(1 -\frac \alpha {6-2\alpha})}  [e^{\rho\vv^\beta} g]_{C^\alpha_k(Q_\theta(z_0))}^{\frac \alpha{6-2\alpha}} \\
& = 
\left( e^{\rho \langle v_0\rangle^{\beta}(\frac 6 \alpha -3)}[e^{\rho \vv^\beta (3-\frac 6 \alpha)} e^{\rho\vv^\beta (\frac 6 \alpha -2)} g]_{C^\alpha_k(Q_\theta(z_0))} \right)^{\frac \alpha {6-2\alpha}}\\
&\leq 
C \left( [e^{\rho \vv^\beta (\frac 6 \alpha - 2)} g]_{C^\alpha_k(Q_\theta(z_0))} + e^{\rho \langle v_0 \rangle^\beta (\frac 6 \alpha - 3)}[e^{\rho \vv^\beta (3-\frac 6 \alpha)}]_{C^\alpha_k(Q_\theta(z_0))}\|e^{\rho \vv^\beta(\frac 6 \alpha - 2)}g\|_{L^\infty(Q_\theta(z_0))}  \right)^{\frac \alpha {6 - 2\alpha}}\\
&\leq 
C \|e^{\rho \vv^\beta(\frac 6 \alpha - 2)}g\|_{C^\alpha_k(Q_\theta(z_0))}^{\frac \alpha {6-2\alpha}},
\end{split}
\]
where we used the fact that $[e^{\rho \vv^\beta (3-\frac 6 \alpha)}]_{C^\alpha_k(Q_\theta(z_0))} \leq C_{\beta,\alpha,\rho} e^{\rho \langle v_0 \rangle^\beta (3-\frac 6 \alpha)}$, since $\beta\leq 1$. Returning to \eqref{e:intermediate-2}, we now have
\[
\begin{split}
\|e^{\rho\vv^\beta} D_v^2g\|_{L^\infty(Q_\theta(z_0))}
 &\leq C  [D_v^2 g]_{C^{2\alpha/3}_k(Q_\theta(z_0))}^{\frac 1 2 - \frac {\alpha/2} {6-2\alpha}} \|e^{\rho\vv^\beta} D_v^2 g\|_{L^\infty(Q_\theta(z_0))}^{\frac 1 2 -\frac {\alpha/2} {6-2\alpha}} \|e^{\rho\vv^\beta (\frac 6 \alpha - 2)} g\|_{C^\alpha_k(Q_\theta(z_0))}^{\frac \alpha{6-2\alpha}}.
 \end{split}
 \]
Absorbing the middle factor on the right into the left-hand side and simplifying, we obtain the conclusion of the lemma.
\end{proof}

\section{Decay estimates for large velocity}\label{s:decay}

In this section, we derive a priori upper bounds that decay sub-exponentially as $|v|\to \infty$. First, we have a preliminary estimate that depends quantitatively on a weaker decay norm of $f$ for positive times.

\begin{lemma}\label{l:sub-exp}
Let $\gamma \in [0,1]$, and let $f$ be a classical solution to the Landau equation \eqref{e:main} on $[0,T]\times\R^6$, periodic in the $x$ variable, such that 
\[
f(0,x,v) \leq K_0 e^{-\rho \vv^\beta}, \quad x\in \R^3, v\in \R^3,
\]
for some $\rho, K_0>0$ and $\beta \in [\gamma,1]\setminus\{0\}$, and such that 
\[
\|f\|_{L^\infty_q([0,T]\times\R^6)} \leq L_0,
\]
for some $q>5+\gamma$ and $L_0>0$. Furthermore, assume that
\begin{equation}\label{e:qual}
e^{\rho \vv^\beta} f(t,x,v) \to 0 \quad \text{ as $|v|\to\infty$, uniformly in $(t,x)$.}
\end{equation}

Then there exists $\sigma>0$, depending only on $\gamma$, $\rho$, $q$, and $L_0$, so that $f$ satisfies 
\[
f(t,x,v) \leq K_0 e^{-(\rho - \sigma t) \vv^\beta}, \quad 0 \leq t \leq \min\left\{T,\frac \rho {2\sigma}\right\}, x\in \R^3, v\in \R^3.
\]
In fact, the conclusion holds for any $\sigma \geq C L_0 \rho(1+\rho)$, where $C>0$ is a constant depending only on $\gamma$ and $q$. 
\end{lemma}

Note that \eqref{e:qual} is a qualitative condition that does not affect any of the constants in the estimate.

\begin{proof}
With $\sigma,\kappa>0$ to be chosen later, and for a small constant $\eps>0$, define the barrier
 \[
 \begin{split}
 \phi(t,v) &= (K_0+\eps) e^{-(\rho - \sigma t) \vv^\beta}.
 \end{split}
 \]
 
We want to show that 
\begin{equation}\label{e:goal}
f(t,x,v) < \phi(t,v), \quad 0\leq t \leq \min\left\{T, \frac \rho {2\sigma}\right\}, x\in \R^3, v\in \R^3.
\end{equation}
If this inequality is false, then we claim there is a point $(t_0,x_0,v_0)$, with $t_0>0$, where $f$ and $\phi$ touch for the first time. Indeed, $f(0,x,v) < \phi(0,v)$ by construction, and because of the condition \eqref{e:qual}, there is some $M>0$ so that $f(t,x,v) < \phi(t,v)$ whenever $|v|>R$. Since $f$ is periodic in $x$, the existence of $z_0 = (t_0,x_0,v_0)$ then follows from the compactness of the domain $[0,T]\times \mathbb T^3 \times \overline {B_R(0)}$ and continuity in $t$.

To keep the notation clean, for the remainder of this proof, evaluations of $f$, $\bar a^f$, and $\bar c^f$ are assumed to be at $z_0$ unless otherwise noted, and evaluations of $\phi$ are at $(t_0,v_0)$.

At the first crossing point $z_0$, we have $\partial_t (\phi - f) \leq 0$, $\nabla_x(\phi - f) = 0$, and $D_v^2 (\phi - f) \geq 0$. These inequalities imply
\begin{equation}\label{e:first-crossing-ineq}
\sigma \vvo^\beta \phi_1 = \partial_t \phi \leq \partial_t f = \tr(\bar a^f D_v^2 f) + \bar c^f f \leq \tr(\bar a^f D_v^2 \phi) + \bar c^f \phi,
\end{equation}
since $\bar a^f$ is non-negative definite. 
To bound the right-hand side in \eqref{e:first-crossing-ineq},  with \eqref{e:a-phi} and Lemma \ref{l:abc}, we have
\[
\begin{split}
&\tr(\bar a^f D_v^2 \phi) + \bar c^f \phi\\
&\leq 
-(\rho - \sigma t_0) \beta \vvo^{\beta-4} \phi \left[ \left((\beta-2) - (\rho-\sigma t_0) \beta \vvo^\beta\right)\bar a^f_{ij} (v_0)_i (v_0)_j + \vvo^2 \tr(\bar a^f)\right] + \bar c^f \phi\\
&\leq 
(\rho - \sigma t_0) \beta C L_0 \vvo^{\gamma+\beta -2} \phi \left( 2 - \beta + (\rho - \sigma t_0) \beta \vvo^\beta\right) + CL_0 \vvo^\gamma \phi\\
&\leq 
(1 + \rho^2) C L_0 \left( \vvo^{\gamma + 2\beta - 2} + \vvo^\gamma\right) \phi \\
&\leq 
C L_0 \rho (1+\rho) \vvo^\gamma \phi,
\end{split}
\]
since $\beta \leq 1$. Here, $C$ is a constant depending only on $\gamma$ and $q$, provided by Lemma \ref{l:abc}. 
Returning to \eqref{e:first-crossing-ineq}, we have, at the crossing point $z_0$, 
\[
\sigma \vvo^\beta \phi  \leq C L_0 \rho (1+\rho) \vvo^\gamma \phi.
\]
Since $\beta \geq \gamma$, this inequality is a contradiction if $\sigma$ is chosen greater than $C L_0 \rho(1+\rho)$. We have established \eqref{e:goal}, and the proof is complete after sending $\eps\to 0$. 
\end{proof}

The previous lemma depends quantitatively on the assumption that $f\in L^\infty_q$ for positive times. To remove this assumption, we need the following technical lemma:

\begin{lemma}\label{l:annoying}\cite[Lemma 2.4]{Tarfulea2020} 
If $L:[0,T]\to [0,\infty)$ is a continuous increasing function and $L(t) \leq A e^{B t L(t)}$ for all $t\in [0,T]$ and some positive constants $A, B$, then
\[
L(t) \leq e A \quad \text{ for } 0 \leq t \leq T_* = \min\left\{ T, \frac 1 {eAB}\right\}.
\]
\end{lemma}

We are now ready to prove our main a priori decay estimate:

\begin{lemma}\label{l:sub-exp2}
Let $\gamma \in [0,1]$, and let $f$ be a classical solution to the Landau equation \eqref{e:main} on $[0,T]\times \R^6$, periodic in the $x$ variable, such that 
\[
f(0,x,v) \leq K_0 e^{-\rho \vv^\beta}, \quad x\in \R^3, v\in \R^3,
\]
for some $\rho, K_0>0$ and $\beta \in [\gamma,1]\setminus\{0\}$. Furthermore, assume that
\begin{equation}\label{e:qual2}
e^{\rho \vv^\beta} f(t,x,v) \to 0 \quad \text{ as $|v|\to\infty$, uniformly in $(t,x)$.}
\end{equation}

Then there exist $T_f, \sigma>0$, depending only on $\gamma$, $\beta$, $\rho$, and $K_0$, so that
\[
f(t,x,v) \leq K_0 e^{-(\rho - \sigma t)\vv^\beta}, \quad 0 \leq t \leq \min\{T, T_f\}, x\in \R^3, v\in \R^3.
\]
\end{lemma}

\begin{proof}
Let $q> \gamma + 5$ be an arbitrary fixed constant. Define the function
\[
L(t) := \|\vv^q f\|_{L^\infty([0,t]\times\R^6)},
\]
and with $t_0\in (0,T]$ fixed, define
\[
T_1 = \frac 1 {2CL(t_0) (1+\rho)},
\]
where $C$ is the constant from Lemma \ref{l:sub-exp}. 
Next, let $n$ be the smallest natural number such that $t\leq n T_1$. Since the function $L(t)$ is increasing, we have $f \leq L(t_0) \vv^{-q}$ when $t\in [0,t_0]$. Applying the upper bound of Lemma \ref{l:sub-exp} with the value $\sigma = C L(t_0) \rho (1+\rho)$ then gives
\[
\vv^q f(t,x,v) \leq K_0 \vv^q e^{-(\rho/2)\vv^\beta}, \quad 0 \leq t \leq \frac {t_0} n,
\]
since $t_0/n \leq T_1$ and $\rho - \sigma T_1 \geq \rho/2$. 
It is straightforward to show, using calculus, that for $s, \rho \geq 0$,
\[
\langle s\rangle^q e^{-\rho \langle s\rangle^\beta} \leq C_{\beta,q} \rho^{-q/\beta},
\]
for some constant $C_{\beta,q}>0$. (Recall $\beta>0$ in all cases.) This inequality implies
\[
\vv^q f(t,x,v) \leq C_{\beta,q} K_0 (\rho/2)^{-q/\beta}, \quad 0 \leq t \leq \frac {t_0} n.
\]
Repeating this argument $n$ times with time shifted appropriately, we obtain
\begin{equation}\label{e:interm}
\begin{split}
\vv^q f(t,x,v) &\leq C_{\beta,q}^n K_0 (\rho/2)^{-(q/\beta) n}\\
&\leq K_0 e^{n C_{\beta,q,\rho}}, \quad 0 \leq t \leq t_0,
\end{split}
\end{equation}
where $C_{\beta,q, \rho} = \log(C_{\beta,q} (\rho/2)^{-q/\beta})$. Taking the supremum over $[0,t_0]\times\R^6$ and using $t_0 = nT_1$, this gives
\[
L(t_0) \leq K_0 e^{C_{\beta,q,\rho} t_0/T_1} = K_0 \exp\left( C_{\beta,q,\rho} 2C (1+\rho) L(t_0) t_0\right).
\]
Since $t_0 \in (0,T]$ was arbitrary, we apply Lemma \ref{l:annoying} to obtain 
\[
L(t) \leq e K_0, \quad 0 \leq t \leq \min\{T, T_f\},\quad  T_f := \frac 1 {e K_0C_{\beta,q,\rho} 2C (1+\rho)}. 
\]
Applying Lemma \ref{l:sub-exp} again with this bound on $L(t)$, we conclude the statement of the lemma.
\end{proof}

\section{Lower bounds}\label{s:lower}

 Recall that in the statement of Theorem \ref{t:existence}, we assume $f$ is strictly positive in a small region:
\[ f_0(x,v) \geq \delta, \quad x\in B_r(x_0), v\in B_r(v_0),\]
for some $r, \delta >0$ and $x_0,v_0\in \R^3$. Our goal is to show that this positive lower bound is propagated forward in time. The statement is similar to \cite[Theorem 1.3]{Henderson-Snelson-Tarfulea2019}, but we need to use a different proof. 

The following lemma is used to spread lower bounds forward in time. The proof is inspired by \cite[Lemma 3.1]{HST2020lowerbounds}, which applied to the Boltzmann equation. 

\begin{lemma}\label{l:pointwise}
Let $f\geq 0$ solve \eqref{e:main} on $[0,T]\times\R^6$, and assume 
\[
\sup_{t,x}\int_{\R^3} (1+|v|^{\gamma+2}) f(t,x,v) \dd v \leq K_0.
\]

 If $f(0,x,v) \geq \delta 1_{\left|x-x_0\right|<r,\left|v-v_0\right|<r/\sigma}$ for some $(x_0, v_0) \in \R^6$ and $\delta, r, \sigma > 0$, then the lower bound
 \[
 f(t,x,v)\ge \frac{\delta}{2}
 \]
holds whenever $0 \leq t\leq \min\{T, \sigma\}$ and, for a constant $C$ depending only on $\gamma$,
\[
\frac{\left|v - v_0\right|^2}{r^2/\sigma^2} + \frac{\left|x - x_0 -tv\right|^2}{r^2} < \frac{1}{4},\quad \text{ and } t < \frac{C K_0^{-1} (r/\sigma)^2}{ \langle \left|v_0 \right| + r/\sigma\rangle^{\gamma +2}}.
\]
\end{lemma}

\begin{proof}
Consider the function

\begin{equation}\label{e:underline}
    \ell(t, x, v):= -c_{1}t + c_{2} \left(1 - \frac{\left|v - v_0\right|^2}{r^2/\sigma^2} - \frac{\left|x - x_0 - tv\right|^2}{r^2}\right)
\end{equation} 
with $c_1, c_2 > 0$ chosen later.  
%

 Assume that $(t,x,v)$ is such that $\ell > 0$. We clearly have $\tr(\bar a^f D_v^2\ell) = \tr (\bar a^f D_v^2(\ell + c_1t))$, so that \eqref{e:underline} implies
\begin{equation}\label{e:dt}
    \begin{split}
   \partial_t \ell &= -c_1 + c_2 \cdot \frac{2}{r^2}(x - x_0 - tv)\cdot v,\\
  v\cdot \nabla_x \ell& = c_2 \frac{2}{r^2} (x - x_0 -tv) \cdot v,\\
      \partial_t \ell + v \cdot \partial_x \ell &= -c_1.
      \end{split}
\end{equation}     
Next, by direct calculation,
\begin{equation}\label{e:d2bound}
\begin{split}
\left|\partial_{ij}\ell(v)\right| &= |4C_2 r^{-4} (\sigma^2(v - v_0) - t (x - x_0 - tv))_i ( \sigma^2(v - v_0) - t(x - x_0 -tv))_j\\
&\quad  + 2 C_2 r^{-2} \delta_{ij}(\sigma^2 + t^2) |\\
& \leq Cc_2 ( (t^2 + \sigma^2)r^{-2} + r^{-2}\sigma^2)\\ 
&\leq Cc_2 \sigma^2 r^{-2}.
\end{split}
\end{equation}
We have used 
$t\leq T$, and that $\ell \leq 0$ if $\frac{\left| v - v_0\right|^2}{r^2/\sigma^2} + \frac{\left|x -x_0 - tv\right|^2}{r^2} > 1$. Using \eqref{e:d2bound}, Lemma \ref{l:abc}, and $\bar c^f \ell \geq 0$,  we have, at points where $\ell$ is positive,
\begin{equation}
    \begin{split}
        Q(f, \ell) &= \tr(\bar a^f D^2 _v \ell) + \bar c^f\ell \\
        &\ge - \left |\tr(\bar a^f D^2 _v \ell) \right |\\
&\ge -C K_0\langle v_0 \rangle ^{\gamma + 2} \sigma^2 r^{-2}.
    \end{split}
\end{equation}
Combining this with \eqref{e:dt} gives
\begin{equation}
    \partial_t \ell + v \cdot \nabla_x \ell = -c_1 < -C K_0 \langle v\rangle^{\gamma +2} \sigma^2 r^{-2} \leq Q(f,\ell), 
\end{equation}
if we make the choice
 \begin{equation}\label{e:choice}
     c_1 = 2CK_0 \langle\left|v_0\right| + r/\sigma\rangle^{\gamma +2} \sigma^2 r^{-2}.
     \end{equation}
 Thus,
 \begin{equation}\label{e:barrier}
 \begin{split}\partial_t \ell + v \cdot \nabla_x \ell = -c_1 < Q(f, \ell) \text{ for all } v \in B_{r} (v_0)
 \end{split}
 \end{equation}

Now, we claim  $f >\ell$ for all $(t,x,v)$ such that $\ell(t,x,v)>0$.   By choosing $c_2 = \frac{3\delta}{4}$, this claim is true when $t=0$.  
 If the claims fails, then there is a first crossing point $(t_0,x_0,v_0)$ with $\ell(t_0,x_0,v_0)>0$, such that $f(t_0,x_0,v_0)=\ell(t_0,x_0,v_0)$ and $f(t,x,v)>\ell(t,x,v)$ whenever $\ell(t,x,v)>0$ and $t<t_0$.
 
 The strict positivity of $t_0$ follows from the compact support of $ \ell(t,\cdot,\cdot)$ for each $t$.  We also have $f(t_0,x,v)\geq \ell(t_0,x,v)$ for all $(x,v)\in \R^6$.
 
 At $(t_0,x_0,v_0)$, we have $\partial_t f \leq \partial_t \ell$ and $\nabla_x f = \nabla_x \ell$, so that \eqref{e:barrier} implies
 \begin{equation}\label{e:less0}
 0\geq (\partial_t + v_0 \cdot \nabla_x)(f-\ell)(t_0,x_0,v_0)>Q(f,f-\ell).
 \end{equation}
 Next, since $f-\ell$ has a local minimum in $v$ at the crossing point, we have
\[ Q(f,f-\ell) \geq \tr(\bar a^f D_v^2(f-\ell))(t_0,x_0,v_0) \geq 0,\]
contradicting \eqref{e:less0}.

 This contradiction implies $f\geq \ell$ whenever $\ell(t,x,v)>0$.  The conclusion then follows by choosing $C$ according to the constant in \eqref{e:choice} and using the definition of $\ell$.
\end{proof}

The purpose of the next lemma is to spread lower bounds to large velocities. This proof uses a barrier argument involving a function like $e^{-c|v|^2/t}$, which is reminiscent of arguments seen in \cite{Desvillettes-Villani2000} and later \cite{Snelson2020}. However, the presence of vacuum regions results in worse coercivity properties for the matrix $\bar a^f$, with the end result that our lower bounds do not extend to arbitrarily large velocities.

\begin{lemma}\label{l:spread-v}
Let $f$ be a solution of the Landau equation \eqref{e:main} on $[0,T]\times\R^6$, such that
\[
\sup_{t,x}\int_{\R^3} f(t,x,v)(1+|v|^{\gamma+2}) \dd v \leq K_0, 
\] 
and for some $\delta, r>0$, $\tau\in (0,1]$, and $x_0,v_0\in \R^3$, assume that $f$ satisfies the lower bound
\[
f(t,x,v) \geq \delta, \quad t\in [0,\tau], x\in B_r(x_m), v\in B_r(v_m).
\]
Then, for any $R>1$, $f$ also satisfies the lower bound
\[
f(t,x,v) \geq  \frac \delta 4  e^{-\kappa t^{-1} |v-v_m|^2}, \quad t\in [0,\tau'], x\in B_{r/4}(x_m + tv_m), v\in B_R(v_m),
\]
where $\kappa>0$ depends on $\delta$, $K_0$, and $r$, and $\tau'  = \min\{\tau, C r/R\}$ for a constant $C>0$ depending on $K_0$.
\end{lemma}

\begin{proof}
First, we recenter around the origin by defining
\begin{equation}\label{e:recenter}
\widetilde f(t,x,v) = f(t, x_m + x + tv_m, v_m + v).
\end{equation}
A direct calculation shows that $\widetilde f$ satisfies the Landau equation in $[0,\tau]\times\R^6$. Our assumptions for $f$ imply 
\begin{equation}\label{e:tilde-f}
\widetilde f(t,x,v) \geq \delta,  \quad t \in \left[0,\widetilde\tau\right], x\in B_{r/2}(0), v\in B_r(0),
\end{equation}
where $\widetilde \tau = \min\{\tau, r/(2|v_m|)\}$. For the remainder of the proof, we write $f$ instead of $\widetilde f$.


Define $\zeta(x) = 1- |x|^2/(r/2)^2$, and note that $\zeta(x) \leq 1_{B_{r/2}}(x)$.  Let $\xi_R:\R^3\to [0,\infty)$ be a smooth, radially decreasing cutoff with $\xi_R = 1$ in $B_R$ and $\xi_R = 0$ outside $B_{2R}$, with $|\nabla \xi_R|\leq C R^{-1}$ and $|D^2 \xi_R| \leq C R^{-2}$ globally in $\R^3$, for some constant $C$. Next, define
\[
\psi(t,x,v) =   \delta ( \zeta(x) - A_1 t)  (\xi_R(v) - A_2 t) e^{-\kappa  t^{-1}|v|^2} ,
\]
where $\kappa, A_1, A_2>0$ are constants to be chosen later. For some small $\eps>0$, we claim that
\begin{equation}\label{e:lower-crossing}
f(t,x,v) > \psi(t,x,v) - \eps, \quad \text{ in $\Omega: = [0,\widetilde\tau]\times\R^3_x \times \left\{|v|\geq r/2\right\}$,}
\end{equation}
where we extend $\psi$ smoothly by zero on $\{t=0\} \times \R^3_x\times \{|v|\geq r/2\}$. 

First, let us show that $f>\psi - \eps$ on the (parabolic) boundary of $\Omega$. When $t=0$ and $|v| \geq r/2$, we have $\psi(0,x,v) = 0$ and  $f(0,x,v)\geq 0 > \psi(0,x,v) - \eps$. 

When $|v| = r/2$ and $t\in [0,\widetilde \tau]$, since $\zeta(x) \leq 1_{B_{r/2}}(x)$ and $\xi_R(v) = 1_{B_r}(v) = 1$, we have
\[
\psi(t,x,v) \leq  \delta  1_{B_{r/2}}(x) \xi_R(v) e^{-\kappa t^{-1} r^2/4}  <  \delta 1_{B_{r/2}}(x) = \delta 1_{B_{r/2}}(x) 1_{B_{r}}(v) \leq  f(t,x,v) < f(t,x,v) + \eps,
\]
where we used \eqref{e:tilde-f} and the fact that $e^{-\kappa t^{-1} r^2/4} < 1$.

Next, we claim that if \eqref{e:lower-crossing} is false, there is a point $z_0 = (t_0,x_0,v_0)$ where $f$ and $\psi-\eps$ cross for the first time, with $t_0>0$. This follows from the fact that $f \geq 0 > \psi - \eps$ for all $(x,v)$ outside of a compact domain. 
Naturally, the crossing point satisfies $|x_0| \leq r/2$ and $|v_0| \leq 2R$.

At the crossing point $z_0$, as above we have
\[
\partial_t( f - \psi) \leq 0, \quad \nabla_x(f-\psi) = 0, \quad D_v^2 (f-\psi) \geq 0,
\]
which implies
\begin{equation}\label{e:lower-bound-psi}
\partial_t \psi + v_0\cdot\nabla_x \psi \geq \partial_t f + v_0\cdot \nabla_x f = \tr(\bar a^f D_v^2 f) + \bar c^f f \geq \tr(\bar a^f D_v^2\psi),
\end{equation}
since $\bar c^f f \geq 0$ and $\bar a^f$ is nonnegative definite. To bound the right side of \eqref{e:lower-bound-psi} from below, we first find via direct calculation
\[
\begin{split}
\partial_{v_i v_j}\psi(t,x,v) =  \delta (\zeta(x) - A_1 t)e^{-\kappa t^{-1} |v|^2 }&\left[  (4\kappa t^{-2} v_i v_j - 2\kappa t^{-1} \delta_{ij}) (\xi_R - A_2 t) \right. \\
& \left.- 2 \kappa t^{-1} \left(  v_j\partial_{v_i} \xi_R + v_i \partial_{v_j} \xi_R \right)  + \partial_{v_i v_j} \xi_R \right].
\end{split}
\]
Therefore, at $z_0$ we have
\[
\begin{split}
\tr(\bar a^f D_v^2 \psi) 
&=
\delta (\zeta(x_0) - A_1 t_0)e^{-\kappa t_0^{-1} |v_0|^2 }\left[  (4\kappa t_0^{-2} \bar a^f (v_0)_i (v_0)_j - 2\kappa t_0^{-1} \tr(\bar a^f))( \xi_R(v_0) - A_2 t_0) \right. \\
&\quad 
\left. - 4 \kappa t_0^{-1}  \bar a^f_{ij} (v_0)_j\partial_{v_i} \xi_R  + \bar a^f_{ij}\partial_{v_i v_j} \xi_R \right].
\end{split}
\]
Since $\xi_R$ is radially decreasing and $\bar a^f$ is positive-definite, we have $-\bar a^f_{ij} (v_0)_j \partial_{v_i} \xi_R \geq 0$ at $z_0$. Next, using Lemmas \ref{l:coercivity} and \ref{l:abc} and $|D_v^2 \xi_R| \leq C R^{-2}$, we have
\begin{equation*}
\begin{split}
\tr(\bar a^f D_v^2 \psi)  & \geq 
\delta ( \zeta(x_0) - A_1 t_0) e^{-\kappa t_0^{-1} |v_0|^2 } \left[\kappa t_0^{-1} \left(c_1  \vvo^\gamma \kappa t_0^{-1} |v_0|^2 - C_2 \vvo^{\gamma+2} \right) (\xi_R(v_0) - A_2 t_0) \right.\\
& \left.  \quad - C_2 \vvo^{\gamma+2}R^{-2}\right],
\end{split}
\end{equation*}
where $c_1$ is the constant from Lemma \ref{l:coercivity}, which depends on $\delta$ and $r$, and $C_2$ is the constant from Lemma \ref{l:abc}, which depends on $K_0$.  Since $|v_0|\geq r/2$, we have $|v_0|^2 \geq \frac r {r+1} \vvo^2$. Therefore, we can choose $\kappa$ sufficiently large, depending only on $c_1$, $C_2$, and $r$, such that $c_1 \kappa t_0^{-1} |v_0|^2 \geq c_1\kappa  |v_0|^2 \geq 2 C_2 \vvo^2$.  Using this, as well as $|v_0|\leq 2R$, we obtain
\begin{equation}\label{e:traf}
\begin{split}
\tr(\bar a^f D_v^2 \psi)  
&\geq \delta (\zeta(x_0) - A_1 t_0) e^{-\kappa t_0^{-1} |v_0|^2}  \left[ c_1\kappa^2 t_0^{-2} R^{\gamma+2}(\xi_R(v_0) - A_2 t_0) - C_2 R^\gamma\right].
\end{split}
\end{equation}


For the left side of \eqref{e:lower-bound-psi}, we have
\[
\begin{split}
\partial_t \psi + v_0\cdot \nabla_x \psi 
&= \delta e^{-\kappa t_0^{-1}|v_0|^2}  \left[(- A_1 + v_0 \cdot \nabla_x \zeta (x_0)) (\xi_R(v_0) - A_2 t)  \right.\\
&\quad + \left.( - A_2 + \kappa t_0^{-2}|v_0|^2(\xi_R(v_0) - A_2 t_0) ) (\zeta(x_0) - A_1 t_0)\right] . 
\end{split}
\]
With $A_1 \geq 4R/r$, we have 
\[
-A_1 + v_0\cdot \nabla_x \zeta(x_0) \leq -A_1 - 2\frac{v_0 \cdot x_0}{r^2} \leq -A_1 + \frac{4 R}{r} \leq 0,
\]
so that 
\[
\partial_t \psi + v_0 \cdot \nabla_x \psi  \leq \delta (\zeta(x_0) - A_1t_0) \left[   4\kappa t_0^{-2} R^2 (\xi_R(v_0) - A_2 t_0) - A_2\right] e^{-\kappa t_0^{-1} |v_0|^2}.
\]
Combining this with \eqref{e:lower-bound-psi} and \eqref{e:traf}, we obtain
\[
\begin{split}
(\zeta(x_0) - A_1 t_0)  & \left[ c_1\kappa^2 t_0^{-2} R^{\gamma+2}(\xi_R(v_0) - A_2 t_0)  - C_2 R^\gamma\right]\\ 
&\leq   (\zeta(x_0) - A_1t_0) \left[   4\kappa t_0^{-2} R^2 (\xi_R(v_0) - A_2 t_0) - A_2\right],
\end{split}
\]
or
\[
 \left[c_1\kappa R^{\gamma} - 4\right] \kappa t_0^{-2} R^2(\xi_R(v_0) - A_2 t_0)  + \left[A_2 - C_2 R^\gamma\right]
\leq 0,
\]
which is a contradiction if $\kappa> 4/(c_1 R^\gamma)$ and $A_2> C_2 R^\gamma$. We have established \eqref{e:lower-crossing}, and after sending $\eps\to 0$, we have shown $f(t,x,v) \geq \psi(t,x,v)$ in $\Omega$. In more detail,
\[
f(t,x,v) \geq \delta (\zeta(x) - A_1 t)(\xi_R(v) - A_2 t) e^{-\kappa t^{-1} |v|^2},
\]
with $A_1$, $A_2$, and $\kappa$ as above. If $z = (t,x,v)$ is such that
\[
|x|\leq r/4, \quad t \leq \min\left\{\frac 1 {4A_1}, \frac 1 {2A_2}\right\}, \quad \text{ and $\frac r 2 \leq |v|\leq R$},
\]
then we clearly have $f(t,x,v) \geq \dfrac \delta 4 e^{-\kappa t^{-1}|v|^2}$. Transforming from $\widetilde f$ back to the original solution $f$, we obtain the statement of the lemma.
\end{proof}

Finally, we prove our main lower bounds result:

\begin{theorem}\label{t:lowerbounds}
Let $\gamma \in [0,1)$, and let  $f:[0,T]\times \R^6 \to [0,\infty)$ be a solution to the Landau equation \eqref{e:main}, periodic in $x$ and satisfying 
\[
f(0,x,v) \geq \delta, \quad |x|< r, |v|< r,
\]
for some $\delta, r>0$, and 
\[
\sup_{t,x} \int_{\R^3} (1+|v|^{\gamma+2}) f(t,x,v) \dd v \leq K_0,
\]
for some $K_0>0$. Then there exists $T' \in (0,T]$ and $R(t)>0$ such that for every $t\in (0,T']$ and $x\in \mathbb R^3$, there exists a $v_x \in B_{R(t)}$ such that 
\[
f(t,x,v) \geq \eta(t) >0, \quad |v-v_x|< r',
\]
where the functions $\eta(t)$ and $1/R(t)$ are uniformly positive on any compact subset of $(0,T']$. The time $T'$, radius $r'$, and the functions $\eta(t), R(t)$ depend on $\delta$, $r$, and $K_0$. 
\end{theorem}

\begin{proof}
Using Lemma \ref{l:pointwise} with $\sigma = 1$ and  $x_0 = v_0 = 0$, we obtain lower bounds of the form
\[
f(t,x,v) \geq \frac \delta 2, \quad t\in [0,\tau], x\in B_{r/3}(0), v\in B_{r/3}(0),
\]
where $\tau\leq C r^2$ for a constant $C>0$ depending on $K_0$. 

Now, let $(t_1,x_1) \in (0,\tau]\times (\R^3\setminus B_r)$ be fixed, and let $v_1 = x_1/t_1$. Lemma \ref{l:spread-v} with $R = \max\{1, 4|v_1|\}$ gives, with $\tau' = \min\{\tau, Cr/R\}$,
\begin{equation}\label{e:int-lower}
f(t,x,v) \geq \frac \delta 8 e^{-\kappa|v|^2/t}, \quad 0< t \leq \tau', |x|\leq r/12, |v|\leq R,
\end{equation}
where $\kappa$ is the constant from Lemma \ref{l:spread-v}, and depends on $\delta$, $r$, and $K_0$. 
Let $t_* = \min\{\tau', t_1/2\}$. 
Letting $v_* = x_1/(t_1 - t_*)$, we then have $|v_*| \leq 2|x_1|/t_1 \leq R/2$. 

For the next step, i.e. spreading lower bounds to the neighborhood of $(t_1,x_1)$, we further restrict the time domain by requiring that $t_1$ satisfy the inequality
\begin{equation}\label{e:t1}
t_1 \leq \min\left\{\left( \frac{ 2C (r/12)^2}{2K_0 R^{\gamma+2}}\right)^{1/3}, \frac {36 C}{K_0 R^\gamma}\right\}.
\end{equation},
where $C$ is the constant from Lemma \ref{l:pointwise}. 
 Since $R = \max\{1,  4|x_1|/t_1\}$, this inequality means that $t_1$ has to be smaller than a constant times $|x_1|^{-(\gamma+2)/(1-\gamma)}$, which we can always guarantee because our $x$ domain $\mathbb T^3_x$ is bounded, resulting in a condition $t_1 \leq \tau''$, where $\tau''$ depends on $K_0$, $r$, and $C$\footnote{In the case of an unbounded spatial domain where $f_0$ satisfies a ``well-distributed'' hypothesis as in \cite{Henderson-Snelson-Tarfulea2019}, $\tau''$ would depend on the maximal distance from any $x$ location to a location where $f_0$ satisfies suitable positive lower bounds.}. This step is the only reason we require $\gamma < 1$. 

We would like to apply Lemma \ref{l:pointwise} to $f(t_* + t, x, v)$, with $x_0 = 0$, $v_0 = v_*$, $r/12$ replacing $r$, and 
\begin{equation}\label{e:sigma}
\sigma = \max\left\{1, \frac r {12} \sqrt{\frac {2C} { t_1 K_0 R^{\gamma+2}} }\right\},
\end{equation}
with $C$ as above. 
We would like the resulting lower bound to hold up to time $t_1/2 \geq t_1 - t_*$. For this, we need to check three conditions:
\begin{itemize}
\item $t_1/2 \leq \sigma$, which holds as a result of \eqref{e:t1}.

\item $r/(12\sigma) \leq R/2$, so that $B_{r/(12\sigma)}(v_*) \subset B_R(0)$. With our choice of $\sigma$, this inequality is equivalent to 
\[
t_1 \leq \frac {36 C}{K_0 R^\gamma},
\]
which is true because of \eqref{e:t1}. 

\item $t_1/2\leq \dfrac{ CK_0^{-1} (r/(12\sigma))^2}{\langle |v_*|+r/(12\sigma)\rangle^{\gamma+2}}$, where $C$ is the constant from Lemma \ref{l:pointwise}. This holds because of \eqref{e:t1} and the previous bullet point. 

\end{itemize}
We are now able to apply Lemma \ref{l:pointwise} with \eqref{e:int-lower} and obtain
\[
f(t_1, x, v) \geq \frac \delta {16} e^{-\kappa R^2/t_*}, \quad \text{ if }\frac{|v-v_*|^2}{r^2/\sigma^2} + \frac{|x- t_1 v|^2}{r^2} < \frac 1 4.
\]
In particular, recalling that $x_1 = t_1 v_1$ and $\sigma \geq 1$, we have that if $|x - x_1|< r/8$ and $|v - v_*|< r/(8\sigma)$, then $|x- t_1 v| < r/4$. Therefore,
\begin{equation}\label{e:t1xv}
f(t_1, x, v) \geq \frac \delta {16} e^{-\kappa R^2/t_*}, \quad |v-v_*|< \frac r {8\sigma}, |x- x_1|< \frac r 8.
\end{equation}


Recall that $R = \max\{1, 4 |x_1|/t_1\}$ and $t_* = \min\{\tau', t_1/2\}$, so that, after tracing the dependence on all constants, we obtain $\eta(t_1)$ and $r'$ as in the statement of the theorem.
 Since $(t_1,x_1)\in \mathbb T^3\times [0,\tau'']$ was arbitrary, the proof is complete. 
\end{proof}

\section{Existence}\label{s:existence}

In this section, we prove our first main result.

\begin{proof}[Proof of Theorem \ref{t:existence}]

First, let us approximate the initial data $f_0$ by a smooth, exponentially decaying function. Let $\psi(x,v)\geq 0$ be a standard smooth mollifier on $\R^6$ satisfying $\int_{B_1(0)} \psi \dd x \dd v = 1$ and $\psi = 0$ outside $B_1(0)$. For any $\eps>0$, let
\[
\psi_\eps(x,v) =  \eps^{-6} \psi(x/\eps,v/\eps).
\]
Also, for any $r>0$, let $\zeta_r(v)$ be a smooth, radially decreasing cutoff function with $\zeta = 0$ outside $B_r$ and $\zeta = 1$ in $B_{r/2}$. Now define the approximate initial data
\[
f_0^\eps(x,v) = \zeta_{1/\eps}(v) [\psi_\eps \ast f_0](x,v),
\]
where $\ast$ denotes convolution in $(x,v)$ variables. The function $f_0^\eps$ is smooth and compactly supported, and by standard arguments, $f_0^\eps \to f_0$ pointwise as $\eps\to 0$. By a direct calculation $f_0^\eps$ is periodic in the $x$ variable with the same period as $f_0$, and can therefore be considered as a function on $\T^3_x\times\R^3_v$.

Applying the main result of \cite{Chaturvedi2019}, quoted above as Theorem \ref{t:chaturvedi}, with initial data $f_0^\eps$, we obtain for each $\eps$ a solution $f^\eps$ to the Landau equation on $[0,T_\eps]\times\R^6$ for some $T_\eps>0$, that is periodic in $x$.

Since $e^{\rho \vv^\beta} f_0^\eps(x,v) \leq 2 e^{\rho \vv^\beta} f_0(x,v) \leq 2K_0$ for $\eps$ small enough, Lemma \ref{l:sub-exp2} implies the upper bound
\begin{equation}\label{e:fe-subexp}
e^{(\rho - \kappa t) \vv^\beta} f^\eps(t,x,v) \leq K_0, \quad t \leq \min\{T_\eps, T_f\},
\end{equation}
for a $T_f>0$ depending only on $\rho$, $\beta$, and $K_0$. Note that the qualitative decay condition \eqref{e:qual2} in Lemma \ref{l:sub-exp2} is satisified because $f^\eps$ satisfies uniform exponential decay $\sim C_\eps e^{-M\vv}$ for any $\eps$ (as a result of Theorem \ref{t:chaturvedi}).

Next, we apply the pointwise lower bounds of Theorem \ref{t:lowerbounds} to $f^\eps$. If $\gamma < 1$, our lower bound assumption for $f_0$ implies that for $\eps$ small enough (depending on $\delta$ and $r$), we have $f_0^\eps(x,v) \geq \delta/2$ when $|x-x_m|< r/2$ and $|v-v_m|< r/2$. With Theorem \ref{t:lowerbounds} applied to a recentered version of $f$ (see e.g. \eqref{e:recenter}), this implies that for each $x$, there is a $v_x\in R(t)$ with 
\begin{equation}\label{e:fe-lower}
f^\eps(t,x,v) \geq \eta(t), \quad |v-v_x|< r', 0 \leq t \leq T_f,
\end{equation}
with $\eta(t)$, $R(t)$, and $r'$ depending only on the initial data (specifically, on $K_0$, $\rho$, $\beta$, $\delta$, $|v_m|$, and $r$). Theorem \ref{t:lowerbounds} also implies that on any compact subset of $(0,T_f]$, we have that $\eta(t)$ is uniformly positive and $R(t)$ is uniformly bounded. 

On the other hand, if $\gamma = 1$, then by assumption, for each $x\in \R^3$, the initial condition $f_0(x,v)$ satisfies positive lower bounds locally near some $v_x\in \R^3$. We apply Lemma \ref{l:pointwise} to propagate these lower bounds to positive times, giving \eqref{e:fe-lower} in this case as well.

We want to show that the approximate solutions $f^\eps$ exist on the time interval $[0,T_f]$, independently of $\eps$. For $\gamma \in (0,1]$, we use the self-generating Gaussian decay estimate \cite[Theorem 1.1]{Snelson2020} and the smoothing estimate \cite[Corollary 1.2]{Snelson2020}, which say: if a solution $f$ satisfies
\begin{equation}\label{e:hydro}
\begin{split}
    &0 < m_0 \leq \int_{\R^3} f(t,x,v)  dv \leq M_0, \quad \int_{\R^3} |v|^2 f(t,x,v) dv \leq E_0,\\
    & \int_{\R^3} f(t,x,v) \log f(t,x,v)  dv \leq H_0,
    \end{split}
\end{equation}
uniformly in $t$ and $x$, then $f$ is $C^\infty$ and satisfies Gaussian decay in $v$, with estimates depending only on $m_0, M_0, E_0$, and $H_0$. For $(t,x) \in [T_\eps/2,T_\eps]\times\R^3$, the three upper bounds in \eqref{e:hydro} follow from \eqref{e:fe-subexp}, and the lower bound on the mass density follows from \eqref{e:fe-lower}. This is enough to re-apply  the existence theorem of \cite{Chaturvedi2019} (which requires Sobolev regularity of order 10, and exponential decay in $v$) with initial data $f^\eps(T_\eps,x,v)$ and repeat until the solution $f^\eps$ has been extended by concatenation to the time interval $[0,T_f]$.

If $\gamma = 0$, the results of \cite{Snelson2020} are not valid, so we must use a different argument to extend $f^\eps$ to the time interval $[0,T_f]$. From the lower bounds \eqref{e:fe-lower} and Lemma \ref{l:coercivity}, and the decay estimate \eqref{e:fe-subexp} with $\beta = 1$, the coefficients $\bar a^{f^\eps}$, $\bar b^{f^\eps}$, and $\bar c^{f^\eps}$ are under control uniformly in $\eps$, so we can bootstrap linear Schauder estimates exactly as in \cite{Henderson-Snelson2020}, giving up a finite number of polynomial moments in $v$ at each step, which can be absorbed into the exponential decay estimate. This shows $f^\eps$ is $C^\infty$ with exponential decay, and we can reapply \cite{Chaturvedi2019} as above. (This step is the only reason we require $\beta = 1$ in the case $\gamma = 0$.)

The lower bound \eqref{e:fe-lower} and the upper bound \eqref{e:fe-subexp} also imply that the bounds \eqref{e:hydro} are satisfied uniformly on any compact subset of $(0,T]\times \R^3_x$, with constants independent of $\eps$. The regularity estimate \cite[Corollary 1.2]{Snelson2020} then implies that for some $\eps_j \to 0$, the sequence $f^{\eps_j}$ converges to a limit $f(t,x,v)\geq 0$ uniformly on compact sets, and that $f$ also satisfies regularity estimates of all orders.

To show the limit $f$ solves the Landau equation, it suffices to use the pointwise convergence of $f^\eps \to f$, $D_v^2 f^\eps \to D_v^2 f$, $(\partial_t + v\cdot\nabla_x)f^\eps \to (\partial_t + v\cdot\nabla_x)f$, and the convergence of the coefficients $\bar a^{f^\eps} \to \bar a^f$ and $\bar c^{f^\eps} \to \bar c^f$, which follows from decay estimates such as \eqref{e:fe-subexp} and the Dominated Convergence Theorem.

Finally, we show that $f$ matches the initial data in the sense of integration against test functions. For any test function $\phi$ as in the statement of Theorem \ref{t:existence}, we integrate $\phi$ against the equation satsfied by $f^\eps$ and obtain
\begin{equation}\label{e:weak}
\int_{\R^6} f_0^\eps(x,v) \phi(0,x,v) \dd v\dd x = \int_0^T \int_{\R^6} [f^\eps (\partial_t + v\cdot\nabla_x)\phi + Q(f^\eps,f^\eps) \phi]\dd v\dd x \dd t.
\end{equation}
Since $\phi$ is compactly supported in $v$ and $f^\eps$ is smooth, we integrate by parts and use  $\nabla_v\cdot \bar b^f = \bar c^f$ and $\sum_j \partial_j \bar a^f_{ij} = - b_i^f$ to write
\[
\begin{split}
\int_{\R^6} Q(f^\eps,f^\eps) \phi \dd v \dd x &= \int_{\R^6} \nabla_v\cdot[\bar a^{f^\eps} \nabla_v f^\eps) + \bar b^{f^\eps}  f^\eps] \phi \dd v \dd x\\
&= -\int_{\R^6} \nabla_v \phi\cdot[\bar a^{f^\eps} \nabla_v f^\eps + \bar b^{f^\eps} f^\eps] \dd v \dd x\\
&= \int_{\R^6} f^\eps [\tr(\bar a^{f_\eps} D_v^2 \phi) -2 \bar b^{f^\eps}\cdot \nabla_v \phi ] \dd v \dd x.
\end{split}
\]
Thanks to \eqref{e:fe-subexp} and Lemma \ref{l:abc}, the coefficients satisfy upper bounds of the form $|\bar a^{f^\eps}| \leq K\vv^{\gamma+2}$, $|\bar b^{f^\eps}| \leq K \vv^{\gamma+1}$, with $K$ independent of $\eps$. Since $\phi$ is compactly supported, we conclude that $\int_{\R^6} Q(f^\eps,f^\eps) \phi\dd v \dd x$ is bounded above independently of $\eps$ and $t$. Therefore, we may send $\eps\to 0$ in \eqref{e:weak}. The collision term $Q(f^\eps,f^\eps) \to Q(f,f)$ pointwise as discussed above, so the proof is complete.
\end{proof}

\section{Propagation of H\"older regularity}\label{s:holder}

In the next two sections, with the goal of proving uniqueness, we place stronger assumptions on the initial data $f_0$ of our solution: for some $\delta, r,R>0$, assume that for all $x\in \R^3$, there is a $v_x\in B_R$ so that
\begin{equation}\label{e:uniform-mass}
f_0(x,v) \geq \delta, \quad v\in B_r(v_x). 
\end{equation}
Furthermore, we assume
\begin{equation}\label{e:f0-exp-upper}
e^{\rho_0 \vv^\beta} f_0 \in L^\infty(\R^6).
\end{equation}
for some $\rho_0>0$ and $\beta\in [\gamma,1]\setminus\{0\}$, and 
\begin{equation}\label{e:Holder-assumption}
e^{\rho \vv^\beta} f_0 \in C^{3\alpha}_{k,x,v}(\R^6),
\end{equation}
for some $\alpha\in (0,1/3)$ and $\rho>0$.

The purpose of the current section is to show that the H\"older continuity that we assume at time zero is propagated forward to positive times. For technical reasons, we work with the specific solution that we constructed above in Theorem \ref{t:existence} (which has an approximating sequence $f^\eps \to f$, with each $f^\eps$ smooth and rapidly decaying) rather than a general classical solution. The precise statement of the result is as follows: 

\begin{theorem}\label{t:holder-prop}
Let $f_0(x,v)\geq 0$ satisfy the lower bound condition \eqref{e:uniform-mass} for some $\delta, r, R>0$, as well as the pointwise upper bound \eqref{e:f0-exp-upper} for some $\rho_0>0$, and the H\"older continuity assumption \eqref{e:Holder-assumption} for some $\alpha \in (0,1/3)$, $\beta\in [\gamma,1]\setminus\{0\}$, and 
\[
\rho = \frac { 4\alpha} {24 - 9\alpha} \rho_0.
\]
Let $f:[0,T]\times\R^6\to [0,\infty)$ be the classical solution to the Landau equation \eqref{e:main} guaranteed by Theorem \ref{t:existence}. 

Then there exists $T_H>0$ such that 
\[
\|e^{(\rho/4)\vv^\beta} f\|_{C^\alpha_{k}([0,\min(T,T_H)]\times \R^6)} \leq C, 
\]
The constants $T_H$ and $C$ depend on $\rho_0$, $\alpha$, $\beta$, $\delta$, $r$, $R$, $\|e^{\rho\vv^\beta} f_0\|_{C^{3\alpha}_{k,x,v}(\R^6)}$, and $\|e^{\rho_0 \vv^\beta} f_0\|_{L^\infty(\R^6)}$. 
\end{theorem}

Following \cite{Tarfulea2020, HST2022boltzmann}, the strategy to prove Theorem \ref{t:holder-prop} is based on bounding a weighted finite difference of $f$ via a barrier argument. For $(t,x,v)\in \R^7$ and $h,k\in B_1(0)\subset \R^3$, and some $\sigma>0$ to be determined, define
\begin{equation}\label{e:g}
\begin{split}
\tau f(t,x,v,h,k) &= f(t,x+h,v+k)\\
\delta f(t,x,v,h,k) &= \tau f(t,x,v,h,k) - f(t,x,v),\\
g(t,x,v,h,k) &= e^{(\rho - \sigma t)\vv^\beta} \frac{ |\delta f(t,x, v)|^2}{(|h|^2 + |k|^2)^{\alpha}},
\end{split}
\end{equation}
where $\rho$ and $\alpha$ are the exponents appearing in \eqref{e:Holder-assumption}.

The function $g$ is chosen in order to control a weighted H\"older seminorm of $f$ in $(x,v)$ variables. This seminorm controlled by $g$ is with respect to the Euclidean distance $d_E((x_1,v_1),(x_2,v_2)) = \sqrt{|x_1 - x_2|^2 + |v_1 - v_2|^2}$, rather than the kinetic distance $d_k$. For any $\Omega\subset \R^6$, let us introduce the notation
\[
[h]_{C^\alpha_{E,x,v}(\Omega)} = \sup_{(x_1,v_1), (x_2,v_2)\in \Omega} \frac{ |h(x_1,v_1) - h(x_2,v_2)|}{d_E((x_1,v_1),(x_2,v_2))^\alpha},
\]
for the H\"older seminorm with respect to the standard Euclidean metric, as well as the norm $\|h\|_{C^\alpha_{E,x,v}(\Omega)} = \|h\|_{L^\infty(\Omega)} + [h]_{C^\alpha_{E,x,v}(\Omega)}$. Euclidean H\"older norms are used only in the current section of this paper. They are needed because of the specific form of the denominator of $g$, which is imposed on us by the proof of Lemma \ref{l:crossing-bounds}. 
A quick calculation shows that 
\begin{equation}\label{e:kin-Euc}
c_1 [h]_{C^\alpha_{k,x,v}(\R^6)} \leq [h]_{C^\alpha_{E,x,v}(\R^6)} \leq C_2[h]_{C^{3\alpha}_{k,x,v}(\R^6)}.
\end{equation}
for constants $c_1, C_2$ depending on $\alpha$. The second inequality here is the reason for the loss of H\"older exponent from $3\alpha$ to $\alpha$ in Theorem \ref{t:holder-prop}, since the proof is based on propagating the $C^\alpha_E$ seminorm of $f$. 

 The key property of the function $g$, that it controls a weighted H\"older seminorm of $f$, is made precise by the following elementary lemma:
\begin{lemma}\label{l:g-controls-H}
For any $\rho, \sigma>0$, $T\in [0,\frac \rho {2\sigma}]$, $\alpha \in (0,1]$, and $\beta\in [0,1]$, the function $g$ defined by \eqref{e:g} satisfies
\[
c_1 \sup_{0\leq t\leq T} [e^{(\rho/2) \vv^\beta} f(t)]_{C^\alpha_{E,x,v}(\R^6)} \leq \|g\|_{L^\infty([0,T]\times\R^6\times B_1(0)^2)} \leq C_2 \sup_{0\leq t\leq T} [e^{\rho \vv^\beta} f(t)]_{C^\alpha_{E,x,v}(\R^6)},
\]
where the constants $c_1$ and $C_2$ depend on $\rho$, $\beta$, and $\alpha$. 
\end{lemma}

A direct calculation shows that $g$ satisfies the equation
\begin{equation}\label{e:g-eqn}
\begin{split}
(\partial_t + v\cdot \nabla_x &+ k\cdot \nabla_h)g + \sigma \vv^\beta g + \frac{ 2\alpha h\cdot k}{|h|^2 + |k|^2}g \\
&= 
 2 \frac{ e^{(\rho-\sigma t)\vv^\beta} \delta f(t,x,v)}{(|h|^2 + |k|^2)^{\alpha}} \left[\tr(\bar a^{\delta f} D_v^2 \tau f + \bar a^f D_v^2 \delta f) + \bar c^{\delta f} \tau f + \bar c^f \delta f\right]
\end{split}
\end{equation}
Using this equation, we will show that $g(t,x,v,h,k)$ is bounded above by a certain barrier function $G(t)$ up to a certain time value. The exact form of $G$ will be dictated by the estimates that are available for $g$ at a first crossing point. Therefore, we derive these estimates first:

\begin{lemma}\label{l:crossing-bounds}
Let $f_0$ be as in Theorem \ref{t:holder-prop}, and let $f$ be a solution of the Landau equation \eqref{e:main} on $[0,T]\times\R^6$ for some $T>0$, with initial data $f_0$. Let $g$ be defined by \eqref{e:g}, with 
\[
\rho = \frac { 4\alpha} {24 - 9\alpha} \rho_0,
\]
where $\rho_0$ is the constant from \eqref{e:f0-exp-upper}.

	There exists $\sigma_0>0$ sufficiently large, such that if $\sigma \geq \sigma_0$, and $\zeta_0 = (t_0,x_0,v_0,h_0,k_0) \in [0,T]\times\R^6\times B_1(0)^2$ is a point with $t_0 \leq \min\left\{ 1, \dfrac \rho {2\sigma}\right\}$ and such that $g(t_0,\cdot)$ achieves its maximum over $\R^6\times \overline{B_1(0)}^2$ at $(x_0,v_0,h_0,k_0)$, then  
\[
\partial_t g \leq C_0\left(g + t_0^{-1 + \mu(\alpha)} g^{1+\nu(\alpha)}\right) \quad {\text{ at $\zeta_0$}},
\]
where $\mu(\alpha)\in (0,1)$ and $\nu(\alpha)>0$ depend only on $\alpha$, and $C_0>0$ and $\sigma_0$ depend on $\rho_0$, $\beta$, $\delta$, $r$, $R$, and $\|e^{\rho_0 \vv^\beta} f_0\|_{L^\infty(\R^6)}$. 
\end{lemma}
\begin{proof}
First, the assumption \eqref{e:f0-exp-upper} for $f_0$ implies, via Lemma \ref{l:sub-exp2}, the decay estimate
\[
\|e^{\rho_0\vv^\beta} f\|_{L^\infty([0,T]\times\R^6)} \leq K_0,
\]
for some $K_0$ depending only on $T$ and the initial data. Here, we may assume $T\leq 1$, since the proof only needs this bound on $f$ up to time $t_0\leq 1$.  Throughout the proof, we absorb dependence on this $K_0$ into constants.

Furthermore, because of our assumption that $ t \leq t_0 \leq \dfrac \rho {2\sigma}$, we have $e^{(\rho - \sigma t_0) \vv^\beta} \geq e^{(\rho/2)\vv^\beta}$ for any $v\in \R^3$. This will be used repeatedly.

For the remainder of this proof, all evaluations of $g$, $\tau f$, and $\delta f$ are assumed to be at $\zeta_0$, and all evaluations of $f$ are assumed to be at $z_0 = (t_0,x_0,v_0)$ unless otherwise noted. We also write $t$ instead of $t_0$, $x$ instead of $x_0$, etc. for equations evaluated at $\zeta_0$, to keep the calculations clean.

Since $\zeta_0$ is a local maximum point of $g$, we have
\begin{equation}\label{e:crossing}
\nabla_x g = \nabla_v g = \nabla_h g = \nabla_k h = 0, \quad D_v^2 g \leq 0, D_k^2 g \leq 0, \quad \text{ at $\zeta_0$.}
\end{equation}
Therefore, evaluating \eqref{e:g-eqn} at $\zeta_0$ results in
\begin{equation}\label{e:g-ineq-alpha}
 \partial_t g + \sigma \vv^\beta g \leq - \frac {2\alpha h\cdot k}{|h|^2+|k|^2} g  +   2 \frac{ e^{(\rho-\sigma t)\vv^\beta} \delta f(t,x,v)}{(|h|^2 + |k|^2)^{\alpha}} \left[\tr(\bar a^{\delta f} D_v^2 \tau f + \bar a^f D_v^2 \delta f) + \bar c^{\delta f} \tau f + \bar c^f \delta f\right]. 
\end{equation}
We want to bound the expression on the right by a power of $g(\zeta_0)$. 
First, we have the simple estimate 
\[
-\frac{2\alpha h\cdot k}{|h|^2+|k|^2} g \leq \alpha  g.
\]
Next, recalling that 
\[
0 = \nabla_v g = \left[\beta (\rho-\sigma t)\vv^{\beta-2} v |\delta f|^2 + 2\delta f \nabla_v(\delta f) \right] \frac {e^{(\rho-\sigma t)\vv^\beta}}{(|h|^2 + |k|^2)^{\alpha}},
\]
we see that
\begin{equation}\label{e:dvdf}
\nabla_v(\delta f) = -\frac 1 2 \beta (\rho-\sigma t) \vv^{\beta-2} v (\delta f).
\end{equation}
We also calculate
\[
\begin{split}
D_v^2 g &= \left[\beta \vv^{\beta-2}(\rho-\sigma t)(I +(\beta-2)v\otimes v \vv^{-2})(\delta f)^2 + \beta^2 (\rho-\sigma t)^2 \vv^{2\beta-4} v\otimes v (\delta f)^2 \right. \\
&\qquad \left. + 4\beta (\rho-\sigma t) \vv^{\beta-2} v \otimes \nabla_v(\delta f) \delta f  + 2\nabla_v(\delta f)\otimes \nabla_v(\delta f) + 2\delta f D_v^2 (\delta f) \right] \frac {e^{(\rho-\sigma t)\vv^\beta}}{(|h|^2 + |k|^2)^{\alpha}},
\end{split}
\]
which implies
\[
\begin{split}
2 \delta f \tr(\bar a^f D_v^2 \delta f) 
& =  \frac {(|h|^2+|k|^2)^{\alpha}} {e^{(\rho-\sigma t)\vv^\beta}}\tr(\bar a^f D_v^2 g) -\beta \vv^{\beta-2} (\rho-\sigma t) [\tr(\bar a^f) + (\beta-2) v\cdot (\bar a^f v) \vv^{-2}] (\delta f)^2 \\
&\quad - \beta^2 (\rho-\sigma t)^2 \vv^{2\beta-4} (\delta f)^2 v\cdot (\bar a^f v) - 4\beta (\rho-\sigma t) \vv^{\beta-2} \delta f  v\cdot (\bar a^f \nabla_v (\delta f))\\
&\quad  - 2 \nabla_v(\delta f)\cdot (\bar a^f \nabla_v(\delta f)).
\end{split}
\]
Since $\bar a^f\geq 0$ and $D_v^2 g\leq 0$ (recall \eqref{e:crossing}), several of the terms in the last expression have a good sign, and we are left with
\[
\begin{split}
2 \delta f \tr(\bar a^f D_v^2 \delta f) 
&\leq \beta (2-\beta) (\rho-\sigma t)\vv^{\beta-4} v\cdot (\bar a^f v) (\delta f)^2 - 4\beta(\rho-\sigma t) \vv^{\beta-2} \delta f v\cdot (\bar a^f\nabla_v(\delta f)).
\end{split} 
\]
For the second term in this right-hand side, we use \eqref{e:dvdf}, yielding
\[
\begin{split}
2 \delta f \tr(\bar a^f D_v^2 \delta f) 
&\leq \beta (2-\beta) (\rho-\sigma t)\vv^{\beta-4} v\cdot (\bar a^f v) (\delta f)^2 + 2\beta^2(\rho-\sigma t)^2 \vv^{2\beta-4} (\delta f)^2 v\cdot (\bar a^fv)\\
&= \left[\beta(2-\beta)(\rho-\sigma t)\vv^{\beta-4} + 2 \beta^2 (\rho-\sigma t)^2 \vv^{2\beta-4}\right] (\delta f)^2 v\cdot (\bar a^f v).
\end{split} 
\]
With Lemma \ref{l:abc}, this gives
\[
2 \delta f \tr(\bar a^f D_v^2 \delta f) \leq C\| f\|_{L^\infty_q([0,t_0]\times\R^6)} \vv^{\beta - 2 + \gamma} (1+\vv^\beta) (\delta f)^2,
\]
for some arbitrarily chosen $q>5+\gamma$, and a constant $C$ depending on $\beta$, $\rho$, and $\gamma$. Absorbing $\|f\|_{L^\infty_q([0,t_0]\times\R^6)} \leq C K_0$ into the constant, we then have
\[
2 \frac{ e^{(\rho-\sigma t)\vv^\beta} \delta f(t,x,v)}{(|h|^2 + |k|^2)^{\alpha}} \tr(\bar a^f D_v^2 \delta f)  \leq C \frac{e^{(\rho-\sigma t)\vv^\beta} } {(|h|^2 + |k|^2)^{\alpha}} \vv^{2\beta - 2 + \gamma} (\delta f)^2 \leq C \vv^{\gamma} g,
\]
since $\beta \leq 1$. 

Next, we address the term in \eqref{e:g-ineq-alpha} with $\tr(\bar a^{\delta f} D_v^2 \tau f)$. Noting that $\delta f = g^{1/2} (|h|^2+|k|^2)^{\alpha/2} e^{-(\rho-\sigma t)\vv^\beta/2}$, we have
\begin{equation}\label{e:a-delta}
\begin{split}
|\bar a^{\delta f}(\zeta_0)|
&\leq C \int_{\R^3} |w|^{\gamma+2} |\delta f(t_0,x_0,v_0-w,h_0,k_0)| \dd w \\
& = C \int_{\R^3} |w|^{\gamma+2} g^{1/2}(t_0,x_0,v_0-w,h_0,k_0) \frac{(|h_0|^2 + |k_0|^2)^{\alpha/2}}{e^{(\rho-\sigma t)\langle v_0-w\rangle^{\beta}/2} } \dd w\\
&\leq C\|g(t_0,\cdot)\|_{L^\infty(\R^6\times B_1^2)}^{1/2} (|h_0|^2 + |k_0|^2)^{\alpha/2} \int_{\R^3}\frac { |w|^{\gamma+2}} {e^{(\rho-\sigma t)\langle v_0-w\rangle^{\beta}/2} } \dd w,\\
&\leq Cg(\zeta_0)^{1/2} (|h_0|^2 + |k_0|^2)^{\alpha/2} \langle v_0\rangle^{\gamma+2},
\end{split}
\end{equation}
for a constant depending on $\rho$. We have used the fact that $g(\zeta_0)$ is the maximum value of $g(t_0,\cdot)$ over $\R^6\times B_1(0)^2$. This implies
\begin{equation}\label{e:a-delta-2}
\begin{split}
2 \frac{ e^{(\rho-\sigma t)\vv^\beta} \delta f}{(|h|^2 + |k|^2)^{\alpha}} \tr(\bar a^{\delta f} D_v^2 \tau f) 
&\leq 2C g(\zeta_0)^{1/2} \langle v_0 \rangle^{\gamma+2} |D_v^2 \tau f(\zeta_0)| \frac{ e^{(\rho-\sigma t)\vv^\beta} \delta f}{(|h_0|^2 + |k_0|^2)^{\alpha/2}}\\
&\leq C g(\zeta_0) \|\vv^{\gamma+2} e^{(\rho-\sigma t_0)\vv^\beta/2} D_v^2 f(t_0,\cdot)\|_{L^\infty(\R^6)}.
\end{split}
\end{equation}
To bound this second-order norm of $f$, we apply Proposition \ref{p:D2f-est} with $\alpha/2$ replacing $\alpha$. We can apply Proposition \ref{p:D2f-est} because of the lower bounds satisfied by $f_0$, which imply suitable lower ellipticity estimates for $\bar a^f$ via Theorem \ref{t:lowerbounds} and Lemma \ref{l:coercivity}. We therefore have
\[
\begin{split}
\|\vv^{\gamma+2} e^{(\rho - \sigma t_0)\vv^\beta/2} & D_v^2 f(t_0,\cdot)\|_{L^\infty(\R^6)}\\
&\leq  
\|e^{(\rho/2) \vv^\beta} D_v^2 f(t_0,\cdot)\|_{L^\infty(\R^6)}\\
& \leq 
Ct_0^{-1-\alpha^2/(24-2\alpha)}  \left(1+ \|e^{(3/\alpha - 1) \rho \vv^\beta} f\|_{C^{\alpha/2}_k([0,t_0]\times\R^6)}^{P(\alpha/2)}\right).
\end{split}
\]
Since the weight $e^{(3/\alpha - 1) \rho \vv^\beta}$ grows too fast for us to close our estimates, we interpolate via Lemma \ref{l:interp1}, then apply Proposition \ref{p:xv-to-t} to translate to a H\"older norm in $(x,v)$ variables: 
\[
\begin{split}
\|\vv^{\beta+2} &e^{(\rho - \sigma t_0)\vv^\beta/2}  D_v^2 f(t_0,\cdot)\|_{L^\infty(\R^6)}\\
&\leq C t_0^{-1-\alpha^2/(24-2\alpha)} \left(1 + \| e^{(\rho/4) \vv^\beta} f\|_{C^{\alpha}_{k}([0,t_0]\times\R^6)}^{P(\alpha/2)/2} \|e^{(6/\alpha - 9/4)\rho \vv^\beta} f\|_{L^\infty([0,t_0]\times\R^6)}^{P(\alpha/2)/2}\right)\\
&\leq Ct_0^{-1-\alpha^2/(24-2\alpha)} \left(1 + \| e^{(\rho/2) \vv^\beta} f\|_{L^\infty([0,t_0],C^{\alpha}_{k,x,v}(\R^6)}^{P(\alpha/2)/2} \right),
\end{split}
\]
where we absorbed $\|e^{(6/\alpha - 9/4)\rho \vv^\beta} f\|_{L^\infty} \leq K_0$ into the constant, since $(6/\alpha - 9/4) \rho = \rho_0$.  
From \eqref{e:kin-Euc} and Lemma \ref{l:g-controls-H}, we have 
\[
\begin{split}
\|e^{(\rho/2) \vv^\beta} f\|_{L^\infty([0,t_0],C^\alpha_{k,x,v}(\R^6))} 
&\leq 
\|e^{(\rho/2) \vv^\beta} f\|_{L^\infty([0,t_0],C^\alpha_{E,x,v}(\R^6))} \\
&\leq 
g(\zeta_0)^{1/2},
\end{split}
\]
since $\zeta_0$ is the location of the maximum value of $g$ over $[0,t_0]\times \R^6\times (B_1(0))^2$. Returning to \eqref{e:a-delta-2}, we have shown
\[
2 \frac{ e^{(\rho-\sigma t)\vv^\beta} \delta f}{(|h|^2 + |k|^2)^{\alpha}} \tr(\bar a^{\delta f} D_v^2 \tau f) 
\leq
C t_0^{-1-\alpha^2/(6-\alpha)}(g(\zeta_0) + g(\zeta_0)^{1+P(\alpha)/4}).
\]

We now address the zeroth-order terms in \eqref{e:g-ineq-alpha}. First, with Lemma \ref{l:abc}, 
\[
2 \frac{ e^{(\rho-\sigma t)\vv^\beta} \delta f}{(|h|^2 + |k|^2)^{\alpha}} \bar c^f \delta f = 2 \bar c^f g \leq C \|f\|_{L^\infty_q([0,t_0]\times\R^6)} \vv^\gamma g \leq C \langle v_0 \rangle^\gamma g,
\]
where $q>5+\gamma$ is arbitrary, and $\|f\|_{L^\infty_q} \leq CK_0$ is absorbed into the constant as above.  Next, proceeding in a similar way to \eqref{e:a-delta}, we have
\[
\begin{split}
|\bar c^{\delta f}(\zeta_0)| &\leq C \int_{\R^3} |w|^\gamma |\delta f(t_0,x_0,v_0-w,h_0,k_0)| \dd w\\
&= C \int_{\R^3} |w|^\gamma \frac{(|h_0|^2 + |k_0|^2)^{\alpha/2}}{e^{(\rho-\sigma t)\langle v_0-w\rangle^{\beta}/2} } \dd w\\
&\leq C g(\zeta_0)^{1/2} (|h_0|^2 + |k_0|^2)^{\alpha/2} \langle v_0\rangle^{\gamma},
\end{split}
\]
and
\[
\begin{split}
2 \frac{ e^{(\rho-\sigma t)\vv^\beta} \delta f}{(|h|^2 + |k|^2)^{\alpha}} \bar c^{\delta f} \tau f &\leq 2C g(\zeta_0)^{1/2} \langle v_0 \rangle^{\gamma} \tau f(\zeta_0) \frac{ e^{(\rho-\sigma t)\vv^\beta} \delta f}{(|h_0|^2 + |k_0|^2)^{\alpha/2}}\\
&\leq C g(\zeta_0) \|\vv^\gamma e^{(\rho-\sigma t)\vv^\beta/2} \tau f\|_{L^\infty([0,T]\times\R^6)}\\
&\leq C g(\zeta_0) \|e^{\rho \vv^\beta/2} f\|_{L^\infty} \leq C g(\zeta_0),
\end{split}
\]
since $\rho/2 \leq \rho_0$. Overall, we have shown that, at the point $\zeta_0$,
\[
\partial_t g + \sigma \vv^\beta g \leq C \vv^\gamma g + Ct_0^{-1-\alpha^2/(24-2\alpha)}(g + g^{1+ P(\alpha/2)/2}),
\]
and since $\gamma \leq \beta$, the proof is complete after choosing $\sigma = C$. 
\end{proof}

We are now ready to prove the main result of this section:

\begin{proof}[Proof of Theorem \ref{t:holder-prop}]
First, we may assume $f$ is smooth and decaying exponentially as $|v|\to \infty$, by passing to the approximating sequence $f^\eps$ from the proof of Theorem \ref{t:existence}. These qualitative properties are used only to get a first crossing point in our barrier argument, so they do not quantitatively affect the estimate we are proving. Therefore, the estimate is preserved in the limit as $\eps\to 0$, since $f^\eps\to f$ pointwise. For simplicity, we write $f$ instead of $f^\eps$ in this proof.

With $g$ defined as in \eqref{e:g} with $\rho$ as in the statement of the theorem, and $\sigma$ as in Lemma \ref{l:crossing-bounds}, we note that
\[
\|g(0,\cdot)\|_{L^\infty(\R^6\times B_1(0)^2)} < \infty,
\]
as a result of \eqref{e:kin-Euc} and our assumption \eqref{e:Holder-assumption} on $f_0$. 

We want to show that $g$ is bounded on some positive time interval. 
To show this, define the barrier $G(t)$ as the solution to the initial value problem
\begin{equation}\label{e:G-ODE}
\begin{cases} 
\partial_t G = 2C_0 (G + t^{-1+\mu(\alpha)} G^{1+\nu(\alpha)}), \\
G(0) = 1 + \|g(0,\cdot)\|_{L^\infty(\R^6\times B_1(0))} + 4\|e^{(\rho/2)\vv^\beta} f\|_{L^\infty([0,T]\times\R^6)}^2,
\end{cases}
\end{equation}
where $C_0>0$, $\mu(\alpha) \in (0,1)$, and $\nu(\alpha)>0$ are the constants from Lemma \ref{l:crossing-bounds}. The norm $\|e^{(\rho/2)\vv^\beta} f\|_{L^\infty}$ is finite because of \eqref{e:f0-exp-upper} and Lemma \ref{l:sub-exp2}.

The solution $G(t)$ to \eqref{e:G-ODE} exists on a time interval $[0,T_G]$, with $T_G$ depending on $\alpha$, $C_0$, and $G(0)$. 
Let $T^* = \min\{T_G,T,\rho/(2\sigma)\}$. We want to show that $g(t,x,v,h,k) < G(t)$ whenever $t\in [0,T^*]$. If this is false, then there must be a point $\zeta_0 = (t_0,x_0,v_0,h_0,k_0)\in [0,T^*]\times\R^6\times\overline{B_1(0)}^2$, with $t_0>0$, where $g$ and $G$ cross for the first time. The existence of this point follows in  a standard way from the compactness of the domain in the $(t,x,h,k)$ variables (recall that $f$ and $g$ are periodic in $x$), as well as the decay of $g$ for large $|v|$ (which follows by the qualitative rapid decay and smoothness of $f$). 

Next, we point out that the crossing point $\zeta_0$ occurs with both $h_0$ and $k_0$ in the interior of $B_1(0)$, since if $|h_0|^2 + |k_0|^2 \geq 1$, the definition \eqref{e:g} of $g$ would imply, since $t_0\leq T^* \leq \rho/(2\sigma)$, 
\[
G(t_0) = g(\zeta_0) \leq e^{(\rho - \sigma t_0)\langle v_0\rangle^\beta} (\delta f)^2(\zeta_0) \leq 4 \|e^{(\rho/2)\langle v_0\rangle^\beta} f\|_{L^\infty([0,T]\times\R^6)}^2 \leq G(0),
\]
which is impossible since $G(t)$ is strictly increasing. 

Since $\zeta_0$ is the first crossing point between $g$ and $G$, we have
\[
\partial_t G(t_0) \leq \partial_t g(\zeta_0),
\]
and because $G$ is independent of $(x,v,h,k)$, the point $(x_0, v_0, h_0, k_0)$ is a maximum point for $g(t_0,\cdot)$, and we can apply Lemma \ref{l:crossing-bounds}:
\[
\partial_t G(t_0) \leq C_0 \left( g(\zeta_0) + t_0^{-1+\mu(\alpha)} g(\zeta_0)^{1+\nu(\alpha)}\right) = C_0 \left( G(t_0) + t_0^{-1+\mu(\alpha)} G(t_0)^{1+\nu(\alpha)}\right)< \partial_t G(t_0),
\]
by \eqref{e:G-ODE}. This contradiction implies $g(t,x,v,h,k) < G(t)$ whenever $t\in [0,T^*]$. There is a time $T_2\in (0, T_G)$ depending on $\alpha$, $C_0$, and $G(0)$, such that $G(t_2) = 2G(0)$. Define
\[
T_H = \min\left\{T_2, \frac \rho {2\sigma}\right\},
\]
and $T^{**} = \min\{T, T_H\}$. 
With Lemma \ref{l:g-controls-H} and Lemma \ref{l:sub-exp2}, we have
\[
\begin{split}
\sup_{0\leq t\leq T^{**}} [e^{(\rho/2)\vv^\beta} f(t)]_{C^\alpha_{E,x,v}(\R^6)}^2
& \leq 
2 G(0) \\
&\leq C \left(\|g(0,\cdot)\|_{L^\infty(\R^6\times B_1(0))} + \|e^{(\rho/2)\vv^\beta} f\|_{L^\infty([0,T^{**}]\times\R^6)}^2\right)\\
&\leq C \left(\|e^{\rho\vv^\beta} f_0\|_{C^\alpha_{E,x,v}(\R^6)}^2 + \|e^{(\rho/2)\vv^\beta} f_0\|_{L^\infty(\R^6)}^2\right).
\end{split}
\]
Next, we translate this inequality to kinetic H\"older norms with \eqref{e:kin-Euc}:
\[
\sup_{0\leq t\leq T^{**}} [e^{(\rho/2)\vv^\beta} f(t)]_{C^\alpha_{k,x,v}(\R^6)}^2 \leq C \|e^{\rho\vv^\beta} f_0\|_{C^{3\alpha}_{k,x,v}(\R^6)}.
\]
Finally, we apply Proposition \ref{p:xv-to-t} to obtain
\[
\| e^{(\rho/4)\vv^\beta} f(t)\|_{C^\alpha_{k,x,v}([0,T^{**}]\times\R^6)}^2 \leq C \|e^{\rho\vv^\beta} f_0\|_{C^{3\alpha}_{k,x,v}(\R^6)}.
\]
as desired. The weight on the left side has been changed to $e^{(\rho/4)\vv^\beta}$ to absorb the polynomial moments lost when applying Proposition \ref{p:xv-to-t}. 

Finally, we recall that $T_H$ depends on $\rho$, $\sigma$, $\alpha$, $C_0$, and $G(0)$. Since $G(0)$ is bounded in terms of $[e^{\rho\vv^\beta} f_0]_{C^{3\alpha}_k(\R^6)}$ and $\|e^{\rho_0 \vv^\beta} f_0\|_{L^\infty(\R^6)}$, the proof is complete.
\end{proof}

\section{Uniqueness}\label{s:uniqueness}

This section establishes the uniqueness of our solutions. First, let us prove an auxiliary lemma:
\begin{lemma}\label{l:ch-bound}
Let $\phi(v) = e^{\rho \vv^\beta}$, with $\beta$ as in Theorem \ref{t:uniqueness}. For any $\mu>0$, there holds
\[
\left\|\frac{h\ast |\cdot|^\mu}{\sqrt \phi}\right\|_{L^2_v(\R^3)} \leq C \|\sqrt \phi h\|_{L^2_v(\R^3)},
\]
whenever $h$ is such that the right-hand side is finite. Here, $C$ is a constant depending only on $\rho$, $\beta$, and $\mu$. 
\end{lemma}
\begin{proof}
From H\"older's inequality, we have
\[
\begin{split}
\int_{\R^3} \frac 1 {\phi} (h\ast |\cdot|^\mu)^2(v) \dd v &= \int_{\R^3} \frac 1 \phi \left(\int_{\R^3} |v-w|^\mu h(w) \dd w\right)^2 \dd v\\
&= \int_{\R^3} \frac 1 \phi \left(\int_{\R^3} \frac{|v-w|^\mu}{\sqrt \phi(w)} \sqrt \phi(w) h(w) \dd w\right)^2 \dd v\\
&\leq C \int_{\R^3} \frac 1 {\phi(v)} \left\| \frac{|v-w|^\mu}{\sqrt \phi(w)}\right\|^2_{L^2_w(\R^3)} \|\sqrt \phi h\|_{L^2(\R^3)}^2 \dd v\\
&\leq C\|\sqrt\phi h\|_{L^2(\R^3)}^2 \int_{\R^3} \frac 1 {\phi(v)} \vv^{2\mu} \dd v \leq C\|\sqrt\phi h\|_{L^2(\R^3)}^2 ,
\end{split}
\]
with $C$ as in the statement of the lemma.
\end{proof}

Now we are ready to prove our main uniqueness result:


\begin{proof}[Proof of Theorem \ref{t:uniqueness}]

The difference $h = f-g$ satisfies the equation
\begin{equation}\label{e:h}
\partial_t h + v\cdot \nabla_x h = Q(f,f) - Q(g,g) = Q(h,f) + Q(g,h).
\end{equation}
Let $\rho \in (0,\rho_0)$ and $\sigma>0$ be constants to be determined later, and define the weight $\phi(t,v) = e^{(\rho-\sigma t) \vv^\beta}$, with $\beta$ as in the statement of the theorem.  We assume throughout the proof that $0 \leq t \leq T_U = \min\{\rho/(2\sigma), T_H\}$, where $T_H$ is the constant from Theorem \ref{t:holder-prop}. In particular, this implies that $\phi \leq e^{(\rho/2)\vv^\beta}$.

As a result of Theorem \ref{t:existence}, we know that $f$ satisfies a uniform decay estimate of the form
\[
e^{\rho \vv^\beta} f(t,x,v) \leq K_0,
\]
for some $K_0, \rho>0$. By assumption, the solution $g$ satisfies
\[
\int_{\R^3} (1+|v|^{\gamma+2}) g(t,x,v) \dd v \leq L_0,
\]
for some $L_0>0$. This is enough to bound the coefficients $\bar a^g$, $\bar b^g$, and $\bar c^g$ from above via Lemma \ref{l:abc}.  In this proof, we absorb $K_0$ and $L_0$ into constants without comment.

Multiplying \eqref{e:h} by $h\phi$ and integrating over $\T^3\times\R^3$, the term $\iint \phi h v\cdot \nabla_x h \dd v \dd x$ vanishes since $\phi$ is independent of $x$, yielding
\begin{equation}\label{e:energy}
\begin{split}
\frac 1 2 \frac d {dt}\|\sqrt\phi h\|_{L^2}^2  + \frac \sigma 2 \int_{\T^3}\int_{\R^3} \vv^\gamma \phi h^2 \dd v \dd x &= \int_{\T^3}\int_{\R^3} \phi h [Q(h,f) + Q(g,h)] \dd v \dd x\\
\end{split}
\end{equation}
Looking at the terms on the right, we start with $\iint \phi h Q(h,f) \dd v\dd x$. For this term, we use the non-divergence form of the collision operator:
\[
\int_{\T^3}\int_{\R^3} \phi h Q(h,f) \dd v \dd x = \int_{\T^3} \int_{\R^3} \phi h \tr(\bar a^h D_v^2 f)\dd v \dd x  + \int_{\T^3}\int_{\R^3}\phi h\bar c^h f \dd v \dd x = I_1 + I_2.
\]
For $I_1$, we first bound $D_v^2 f$ by combining Proposition \ref{p:D2f-est} and Theorem \ref{t:holder-prop} with $\alpha/3$ replacing $\alpha$:
\[
\begin{split}
\|(\phi D_v^2 f)(t) \|_{L^\infty(\T^3\times\R^6)} 
&\leq 
C \|e^{(\rho/2)\vv^\beta} D_v^2 f(t)\|_{L^\infty(\T^3\times\R^6)}  \\
&\leq 
C(1+t^{\kappa(\alpha)}) \|e^{\rho' \vv^\beta} f\|_{C^{\alpha/3}_k([0,t]\times\R^6)} \\
&\leq   C(1+t^{\kappa(\alpha)}), 
\end{split}
\]
where $\kappa(\alpha) = -1 + (\alpha/3)^2/(6-\alpha/3)$ and $\rho' = (9/\alpha - 1) \rho$.  From Theorem \ref{t:holder-prop}, we see that the constant $C$ depends on $\|e^{\rho'' \vv^\beta} f_0\|_{L^\infty(\R^6)}$, where $\rho'' = (24-9\alpha)/(\alpha) \rho'$. We choose the parameter $\rho$ in the definition of our weight $\phi$ by setting $\rho'' = \rho_0$ and solving for $\rho$. 

Returning to the term $I_1$, we bound $\bar a^h$ using Lemma \ref{l:ch-bound} with $\mu = \gamma + 2$, giving
\[
\begin{split}
\int_{\T^3}\int_{\R^3} \phi  h \tr( \bar a^h D_v^2 f) \dd v \dd x &\leq \|(\phi D_v^2 f)(t) \|_{L^\infty(\T^3\times \R^6)} \int_{\T^3} \int_{\R^3} \sqrt \phi h \frac{|\bar a^h|}{\sqrt \phi} \dd v \dd x\\
&\leq C(1 + t^{\kappa(\alpha)}) \int_{\T^3} \|\sqrt \phi h\|_{L^2_v(\R^3)} \left\|\frac{|\bar a^h|}{\sqrt \phi}\right\|_{L^2_v(\R^3)} \dd x\\
&\leq C(1 + t^{\kappa(\alpha)}) \int_{\T^3} \|\sqrt \phi h\|_{L^2_v(\R^3)}^2 \dd x\\
&\leq C(1 + t^{\kappa(\alpha)})\|\sqrt \phi h\|_{L^2(\T^3\times\R^3)}^2.
\end{split}
\]
For $I_2$, we have
\[
\begin{split}
\int_{\T^3}\int_{\R^3} \phi h \bar c^h f \dd v \dd x &\leq \|(\phi f)(t)\|_{L^\infty(\T^3\times\R^3)}\int_{\T^3} \int_{\R^3} \sqrt \phi h \frac{\bar c^h}{\sqrt \phi} \dd v \dd x\\
&\leq \|(\phi f)(t)\|_{L^\infty(\T^3\times\R^3)}\int_{\T^3} \|\sqrt \phi h\|_{L^2_v(\R^3)} \left\|\frac{\bar c^h}{\sqrt \phi}\right\|_{L^2_v(\R^3)} \dd x\\
&\leq C\|\phi f\|_{L^\infty} \int_{\T^3} \|\sqrt \phi h\|_{L^2_v(\R^3)}^2 \dd x = C \|\phi f\|_{L^\infty}  \|\sqrt \phi h\|_{L^2}^2,
\end{split}
\]
by Lemma \ref{l:ch-bound} with $\mu = \gamma$.

Next, we address the term $\iint \phi h Q(g,h) \dd v \dd x$ in \eqref{e:energy}. Here, we use the divergence form of the collision operator:
\[
\int_{\T^3}\int_{\R^3} \phi h Q(g,h) \dd v \dd x = \int_{\T^3} \int_{\R^3} \phi h [\nabla_v\cdot (\bar a^g \nabla_v h) + \bar b^g h +   \bar c^g h] \dd v \dd x = J_1 + J_2 + J_3.
\]
%
%
For $J_1$, we integrate by parts in $v$:
\[
J_1 =  - \int_{\T^3}\int_{\R^3} [\phi \nabla_v h + h \nabla_v \phi] \cdot (\bar a^g \nabla_v h) \dd v \dd x.
\]
Since $\bar a^g$ is a postive-definite matrix, we have $(\nabla_v h + \eps \nabla_v\phi)\cdot [\bar a^g (\nabla h + \eps \nabla_v \phi)] \geq 0$ for any $\eps>0$, which implies the following inequality:
\[
-\nabla_v \phi\cdot (\bar a^g \nabla_v h) \leq \frac \eps 2 \nabla_v \phi\cdot (\bar a^g \nabla_v \phi) + \frac 1 {2\eps} \nabla_v h\cdot (\bar a^g \nabla_v h).
\]
Choosing $\eps = h/(2\phi)$, this gives
\begin{equation}\label{e:good-term}
\begin{split}
-\int_{\T^3}\int_{\R^3} h \nabla_v \phi \cdot (\bar a^g \nabla_v h) \dd v \dd x &\leq  \int_{\T^3}\int_{\R^3} \phi \nabla_v h \cdot (\bar a^g \nabla_v h) \dd v \dd x + \frac 1 4 \int_{\T^3}\int_{\R^3}   h^2 \frac{\nabla_v \phi}{\phi}\cdot(\bar a^g \nabla_v \phi) \dd v \dd x,
\end{split}
\end{equation}
and
\[
J_1 \leq  \frac 1 4 \int_{\T^3}\int_{\R^3}   h^2 \frac{\nabla_v \phi}{\phi}\cdot(\bar a^g \nabla_v \phi) \dd v \dd x.
\]
From the upper bound for $\bar a^g$ in Lemma \ref{l:abc}, we have
\[
\begin{split}
J_1 &\leq  \frac {\beta^2} 4 \int_{\T^3} \int_{\R^3} h^2 \phi \vv^{2\beta-4} v\cdot (\bar a^g v) \dd v \dd x \\
& \leq  C L_0  \int_{\T^3} \int_{\R^3} h^2 \phi \vv^{2\beta + \gamma-2} \dd v \dd x \leq C \|\sqrt \phi h\|_{L^2_{\gamma/2}(\T^3\times\R^3)}^2.
\end{split}
\]
For $J_2$, the growth of $\bar b^g\approx \vv^{\gamma+1}$ presents a difficulty, since the good term on the left side of \eqref{e:energy} only allows us to absorb a weight like $\vv^\beta \phi$. To get around this, we integrate by parts repeatedly and use the facts that $\bar b^g_i = -\sum_{j=1}^3 \partial_{v_j}\bar a^g_{ij}$ and $\nabla_v\cdot \bar b^g = \bar c^g$:
\[
\begin{split}
J_2
 &= 
 \frac 1 2 \int_{\T^3}\int_{\R^3} \phi  \bar b^g \cdot \nabla_v (h^2) \dd v \dd x\\
&= 
\frac 1 2 \int_{\T^3} \int_{\R^3} h^2  \left(\sum_{j=1}^3 \partial_{v_j}\bar a^g_{ij}\right) \cdot \nabla_v \phi\dd v \dd x - \frac 1 2 \int_{\T^3}\int_{\R^3} \phi \bar c^f h^2 \dd v \dd x \\
&= 
- \int_{\T^3} \int_{\R^3} h \nabla_v \phi \cdot (\bar a^g \nabla_v h) \dd v \dd x - \frac 1 2 \int_{\T^3}\int_{\R^3} h^2 \tr(\bar a^g D_v^2 \phi) \dd v \dd x - \frac 1 2 \int_{\T^3}\int_{\R^3} \phi \bar c^g h^2 \dd v \dd x.
\end{split}
\]
In this right-hand side, the first term is equal to the term handled above in \eqref{e:good-term}, and we estimate it in the same way. The third term is equal to $-\frac 1 2 J_3$. For the middle term, we use the expression \eqref{e:a-phi} (with $\rho - \sigma t$ replacing $\rho$) for $\tr(\bar a^g D_v^2 \phi)$ and discard negative terms to obtain
\[
\begin{split}
- \frac 1 2 &\int_{\T^3}\int_{\R^3} h^2 \tr(\bar a^g D_v^2 \phi) \dd v \dd x \\
&= 
- \frac {(\rho-\sigma t) \beta } 2 \int_{\T^3}\int_{\R^3} h^2 \phi \vv^{\beta-4}[ \left( (\beta-2) +(\rho-\sigma t) \beta \vv^\beta\right) v\cdot (\bar a^g v) + \vv^2 \tr(\bar a^g)] \dd v \dd x\\
&\leq 
\frac {(\rho-\sigma t) \beta (\beta-2)} 2 \int_{\T^3}\int_{\R^3} h^2 \phi \vv^{\beta-4}v\cdot (\bar a^g v) \dd v \dd x\\
&\leq C L_0 \int_{\T^3}\int_{\R^3} \vv^{\beta +\gamma - 2} h^2 \phi \dd v \dd x = C \|\sqrt \phi h\|_{L^2(\T^3\times\R^3)}^2,
\end{split}
\]
after using Lemma \ref{l:abc}. 
%
For $J_3$, Lemma \ref{l:abc} implies
\[
\int_{\T^3}\int_{\R^3} \phi h^2 \bar c^g \dd v \dd x \leq C L_0 \int_{\T^3} \int_{\R^3} \vv^\gamma \phi h^2 \dd v \dd x \leq C \|\sqrt \phi h\|_{L^2_{\gamma/2}(\T^3\times\R^3) }^2.
\]

Putting everything together, we have
\[
\frac 1 2 \|\sqrt\phi h\|_{L^2}^2  + \frac \sigma 2 \int_{\T^3}\int_{\R^3} \vv^\beta \phi h^2 \dd v \dd x \leq C \left( 1 + t_0^{\kappa(\alpha)}\right) \|\sqrt \phi h\|_{L^2(\T^3\times\R^3)}^2 + C \|\sqrt \phi h\|_{L^2_{\gamma/2}(\T^3\times\R^3)}^2.
\]
The constant in this inequality depends only on $\alpha$, $\rho_0$, $K_0$, and $L_0$. 
Choosing $\sigma = C$ and applying Gronwall's inequality, we obtain (since $\beta\geq \gamma$)
\[
\|\sqrt \phi(t) h(t)\|_{L^2(\T^3\times\R^3)} \leq \|\sqrt \phi(0) h(0)\|_{L^2(\T^3\times\R^3)} \exp(C(t_0 + t_0^{1+ \kappa(\alpha)}) = 0,
\]
so that $h(t,x,v) = 0$ for all $t$ and almost every $(x,v)$. By continuity, $h(t,x,v) = 0$ everywhere, and we conclude $f=g$ pointwise. Since $T_U$ depends on $\rho$, $\sigma$, and $T_H$, we conclude the proof.
\end{proof}
\appendix

\section{Regularity theory}\label{s:a}

\subsection{Change of variables}

When applying regularity estimates for the Landau equation, the ellipticity of the matrix $\bar a^f$ degenerate for large $v$. A change of variables was developed in \cite{cameron2017landau} to precisely track this degeneration. 

For a fixed $z_0 = (t_0,x_0,v_0) \in [\tau,T]\times\R^6$, if $|v_0|>2$, let $S$ be the linear transformation defined by
\[
S\xi = 
\begin{cases} 
|v_0|^{1+\gamma/2} \xi, &\xi \perp v_0,\\
|v_0|^{\gamma/2} \xi, & \xi \parallel v_0.
\end{cases}
\]
If $|v_0|\leq 2$, then we define $S$ as the identity matrix. 
Next, define
\[
\mathcal T_{z_0}(t,x,v) = (t_0+t,x_0 + Sx + t v_0,v_0+Sv).
\]
Given a solution to the Landau equation \eqref{e:main} on $[0,T]\times\R^6$, one then defines
\begin{equation}\label{e:r1}
r_1 = \begin{cases}
|v_0|^{-1-\gamma/2} \min\left(1, \sqrt{t_0/2}\right), & |v_0|>2,\\
\min\left(1, \sqrt{t_0/2}\right), & |v_0|\leq 2,
\end{cases}
\end{equation}
and
\[
f_{z_0}(t,x,v) := f(\mathcal T_{z_0}(\delta_{r_1}(z))), \quad z\in Q_1(0),
\]
where $\delta_{r_1}(z) = (r_1^2t,r_1^3x,r_1v)$. By direct calculation,  $f_{z_0}$ satisfies 
both the divergence form equation
\begin{equation}\label{e:AC-div}
\partial_t f_{z_0} + v\cdot \nabla_x f_{z_0} = \nabla_v \cdot (A(z) \nabla_v f_{z_0}) + B(z)\cdot \nabla_v f_{z_0} + C(z) f_{z_0},
\end{equation}
and the nondivergence-form equation
\begin{equation}\label{e:AC-nondiv}
\partial_t f_{z_0} + v\cdot \nabla_x f_{z_0} = \tr(A(z) D_v^2 f_{z_0})  + C(z) f_{z_0},
\end{equation}
in $Q_{1}(0)$, where the coefficients are defined by
\begin{equation}\label{e:new-coeffs}
\begin{split}
A(z) &= S^{-1} \bar a^f(\mathcal T_{z_0} (\delta_{r_1}(z))) S^{-1},\\
B(z) &= r_1 S^{-1} \bar b^f(\mathcal T_{z_0} (\delta_{r_1}(z)))\\
C(z) &= r_1^2 \bar c^f(\mathcal T_{z_0} (\delta_{r_1}(z))).
\end{split}
\end{equation}
The key properties of $f_{z_0}$ and the transformed equation are contained in the following lemma, which first appeared in \cite{cameron2017landau} and was originally derived for the case $\gamma \in (-2,0)$. However, as pointed out in \cite{Snelson2020}, the result extends to the case $\gamma >0$ with essentially the same proof. The lemma is as follows:
\begin{lemma}[\cite{cameron2017landau}]\label{l:COV}
With $z_0 \in (0,T]\times \R^6$ given, let $S$ and $\mathcal T_{z_0}$ be defined as above.

\begin{enumerate}

\item[(a)] There exists a constant $C>0$ independent of $z_0$, such that 
\[
C^{-1} |v_0| \leq |v_0+r_1Sv| \leq C|v_0|, \quad v\in B_1(0).
\]
\item[(b)]
Let $f$ be a solution of the Landau equation on $[0,T]\times\R^6$, satisfying
\begin{equation}\label{e:sim}
\bar a^f(t,x,v) \xi_i \xi_j \approx  
\begin{cases} 
\vv^{\gamma}|\xi|^2, & \xi \perp v,\\
\vv^{\gamma+2} |\xi|^2, & \xi \parallel v,
\end{cases}
\quad |\bar b^f(t,x,v)|\lesssim \vv^{\gamma+1}, \quad \bar c^f(t,x,v) \lesssim \vv^\gamma.
\end{equation}
Then the coefficients $A$, $B$, and $C$ defined in \eqref{e:new-coeffs} satisfy
\begin{equation}\label{e:ABCbounds}
\begin{split}
\lambda I &\leq A(z) \leq \Lambda I,\\
B(z) &\leq \Lambda \langle v_0\rangle^{1+\gamma/2},\\
C(z) &\leq \Lambda \langle v_0\rangle^\gamma,
\end{split}
\end{equation}
for all $z\in Q_{1}(0)$, where $\lambda$ and $\Lambda$ are constants depending only on the implied constants in \eqref{e:sim}.
\end{enumerate}
\end{lemma}

We also note the following properties of our change of variables, which can be verified by a direct computation: for any $z_1,z_2\in Q_1$,
\begin{equation}\label{e:dk-scale}
d_k(\delta_{r_1}(z_1),\delta_{r_1}(z_2)) = r_1 d_k(z_1,z_2),
\end{equation}
and
\begin{equation}\label{e:distance-compare}
\min\{1,\sqrt{t_0/2}\} \langle v_0 \rangle^{-1} d_k(z_1,z_2) \leq d_k(\mathcal T_{z_0}(\delta_{r_1}(z_1)),\mathcal T_{z_0}(\delta_{r_1}(z_2)))\leq \min\{1,\sqrt{t_0/2}\} d_k(z_1,z_2).
\end{equation}
The following lemma relates the regularity of $f_{z_0}$ to the regularity of $f$:
\begin{lemma}\label{l:f-and-fz}
 Let $f:[0,T]\times \R^6 \to \R$ and $z_0\in (0,T]\times\R^6$ be given. Let $r_1$ be defined by \eqref{e:r1}, and let  $r_0 = \min\{\sqrt{t_0/2}, 1\}$
\begin{enumerate}
\item[(a)] If $f\in L^\infty_q(Q_{r_0}^{t,x}(z_0)\times\R^3)$ for some $q>0$, then 
\[
\|f_{z_0}\|_{L^\infty(Q_1)} \leq C\langle v_0 \rangle^{-q} \|f\|_{L^\infty_q(Q_{r_0}^{t,x}(z_0)\times\R^3)}.
\]

\item[(b)] If $f_{z_0}\in C^\alpha_k(Q_\theta)$ for some $\alpha, \theta\in (0,1]$, then
\[
[f]_{C^\alpha_k(Q_{r_1 \theta}(z_0))} \leq C \max\{1, t_0^{-\alpha/2}\} \langle v_0\rangle^{\alpha} [f_{z_0}]_{C^\alpha_k(Q_\theta)}.
\]

\item[(c)] If $f\in C^\alpha_k(Q_\theta(z_0))$, for some $\alpha,\theta\in (0,1]$, then
\[
[f_{z_0}]_{C^\alpha_k(Q_\theta)} \leq C\min\{1,t_0^{\alpha/2}\} [f]_{C^\alpha_k(Q_{\theta}(z_0))}.
\]
\end{enumerate}
In all three estimates, the constant $C>0$ is independent of $f$ and $z_0$.
\end{lemma}

\begin{proof}
First, we note that note that $Q_{r_1 \theta}(z_0) \subset \mathcal T_{z_0}(\delta_{r_1}(Q_\theta))\subset Q_{r_0\theta}(z_0)$. Conclusion (a) then follows from Lemma \ref{l:COV}(a), and conclusions (b) and (c) follow by applying \eqref{e:distance-compare}.
\end{proof}

The purpose of the following lemma is to pass regularity of $f$ to the coefficients $A$ and $C$. Since $A$ and $C$ are nonlocal in the $v$ variable, the assumption of H\"older continuity for $f$ must be made on the entire velocity domain $\R^3$. The proof of this lemma is the same as \cite[Lemma 3.3]{Henderson2020} or \cite[Lemma 2.7]{Tarfulea2020}.
\begin{lemma}\label{l:ABC-Holder}
Let $f$ be defined in $\Omega\times\R^3$ for some $(t,x)$ domain $\Omega\subset \R^4$, and let  $z_0$ be such that $Q_{r_1}(z_0)\subset \Omega\times \R^3$. Assume that $\vv^m f\in C^\alpha_k(\Omega\times\R^3)$ for some $\alpha\in (0,1)$ and $m>5 + \gamma + \alpha/3$. 

Then the coefficients $A$ and $C$ defined in \eqref{e:new-coeffs} are H\"older continuous in $Q_1$, and
\begin{equation}\label{e:AC-holder}
\begin{split}
[A]_{C^{2\alpha/3}_k(Q_1)} &\leq C\langle v_0\rangle^{2 + \alpha/3} [\vv^m f]_{C^\alpha_k(\Omega\times\R^3)},\\
[C]_{C^{2\alpha/3}_k(Q_1)} &\leq C\langle v_0\rangle^{\alpha/3} [\vv^m f]_{C^\alpha_k(\Omega\times\R^3)}.
\end{split}
\end{equation}
The constant $C$ depends only on $\gamma$, $\alpha$, and $m$. 
\end{lemma}

\subsection{Regularity estimates}

The following lemma is a Schauder estimate for the Landau equation. Because of the nonlocality of the coefficients, this estimate depends on the H\"older continuity of $f$ over the entire velocity domain.

\begin{lemma}\label{l:schauder}
Let $f$ be a solution to the Landau equation on $[0,T]\times \R^6$. Let $z_0\in (0,T]\times\R^6$ be given, and define
\[
\Omega(z_0) = (Q_{r_0}^{t,x}(z_0) \times \R^3_v),
\]
where $r_0 = \min\{1, \sqrt{t_0/2}\}$. For some $q$, $m$, and $\alpha$ with $\alpha \in (0,1)$ and
\[
q> m > 5+\gamma + \alpha/3,
\]
assume that
\begin{equation}\label{e:f-Lq}
f(t,x,v) \leq K_0\vv^{-q}, \quad \text{in $\Omega(z_0)$,}
\end{equation}
and $\vv^m f \in C^\alpha_k(\Omega)$. Assume further that
\begin{equation}\label{e:af-lower-bound0}
\bar a^f_{ij}(t,x,v)\xi_i \xi_j \geq
 \lambda_0 \begin{cases}
\vv^\gamma, &\xi \perp v,\\
\vv^{\gamma+2}, & \xi \parallel v,
\end{cases}
\quad \text{for all $(t,x,v)\in \Omega(z_0)$ and $\xi\in \R^3$.}
\end{equation}
Then
\[
\begin{split}
[D_v^2 f]_{C^{2\alpha/3}_k(Q_{r_1/2}(z_0))} +& [(\partial_t  +v\cdot \nabla_x) f]_{C^{2\alpha/3}_k(Q_{r_1/2}(z_0))}\\
& \leq C(1+t_0^{-1-\alpha/3}) \langle v_0\rangle^{-(q+2m)/3+9+3\alpha + 6/\alpha + 2\alpha^2/9 + \gamma} \left( 1 + [\vv^m f]_{C^\alpha_k(\Omega(z_0))}^{3+2\alpha/3+3/\alpha}\right),
\end{split}
\]
where $C>0$ is a constant depending only on $K_0$ and $\lambda_0$, and $r_1$ is defined in \eqref{e:r1}.
\end{lemma}

\begin{proof}
Defining $f_{z_0}$ as above, with base point $z_0$, we will work with the nondivergence-form equation \eqref{e:AC-nondiv}. Our first step is to verify the hypotheses of Lemma \ref{l:COV}. From Lemma \ref{l:abc}, our assumption \eqref{e:f-Lq} provides suitable upper bounds on $\bar c^f$ as in \eqref{e:sim}. The lower bound on $\bar a^f$ in \eqref{e:sim} follows from \eqref{e:af-lower-bound0}, and the upper bound follows from Lemma \ref{l:abc}. Therefore, the bounds \eqref{e:ABCbounds} are valid for the coefficients $A$ and $C$ defined in \eqref{e:new-coeffs}, with constants depending only on $\lambda_0$ and $K_0$.

The Schauder estimate of \cite[Theorem 2.9]{Henderson2020}, with $2\alpha/3$ replacing $\alpha$, yields
\[
\begin{split}
[D_v^2 f_{z_0}]_{C^{2\alpha/3}_k(Q_{1/2})} + &[(\partial_t  +v\cdot \nabla_x) f_{z_0}]_{C^{2\alpha/3}_k(Q_{1/2})}\\
& \leq C \left( [C f_{z_0}]_{C^{2\alpha/3}_k(Q_1)} + \|A\|_{C^{2\alpha/3}_k(Q_1)}^{3+2\alpha/3+3/\alpha} \|f_{z_0}\|_{L^\infty(Q_1)} \right).
\end{split}
\]
Applying Lemma \ref{l:product} for the product $C f_{z_0}$, and using Lemma \ref{l:ABC-Holder} and the upper bounds on $A$ and $C$ from Lemma \ref{l:COV}, we have
\begin{equation}\label{e:Dv2f-est}
\begin{split}
[D_v^2 f_{z_0}]_{C^{2\alpha/3}_k(Q_{1/2})} + [(\partial_t  +v\cdot &\nabla_x) f_{z_0}]_{C^{2\alpha/3}_k(Q_{1/2})}\\
& \leq C \left( \|f_{z_0}\|_{L^\infty(Q_1)} \langle v_0\rangle^{\alpha/3}[\vv^m f]_{C^\alpha_k(\Omega(z_0))} + [f_{z_0}]_{C^{2\alpha/3}_k(Q_1)} \langle v_0\rangle^\gamma\right.\\
&\left. \qquad + \left( \langle v_0\rangle^{2+\alpha/3}[\vv^m f]_{C^\alpha_k(\Omega(z_0))}\right) ^{3+2\alpha/3+3/\alpha} \|f_{z_0}\|_{L^\infty(Q_1)} \right).
\end{split}
\end{equation}
Next, we use Lemma \ref{l:f-and-fz}(c), the interpolation $[f]_{C^{2\alpha/3}_k(Q_{r_0}(z_0))} \leq C[f]_{C^\alpha_k(Q_{r_0}(z_0))}^{2/3}\|f\|_{L^\infty(Q_{r_0}(z_0))}^{1/3}$, and Lemma \ref{l:product} to write
\[
\begin{split}
[f_{z_0}]_{C^{2\alpha/3}_k(Q_1)} &\leq C\min\{1,t_0^{\alpha/3}\} [f]_{C^{2\alpha/3}_k(Q_{r_0}(z_0))}\\
&\leq C[f]_{C^\alpha_k(Q_{r_0}(z_0))}^{2/3}\|f\|_{L^\infty(Q_{r_0}(z_0))}^{1/3}\\
&\leq C \langle v_0\rangle^{-(q+2m)/3} [\vv^m f]_{C^\alpha_k(Q_{r_0}(z_0))}^{2/3}\|f\|_{L^\infty_q(\Omega(z_0))}^{1/3}.
\end{split}
\]
Returning to \eqref{e:Dv2f-est}, using $\|f_{z_0}\|_{L^\infty(Q_1)}\lesssim \langle v_0\rangle^{-q}\|f\|_{L^\infty_q(\Omega(z_0))}$, absorbing the norm $\|f\|_{L^\infty_q}$ into the constant, and keeping only the largest powers of $\langle v_0 \rangle$ and $[\vv^m f]_{C^\alpha_k}$, we obtain 
\begin{equation}\label{e:intermediate-result}
\begin{split}
[D_v^2 f_{z_0}]_{C^{2\alpha/3}_k(Q_{1/2})} +& [(\partial_t  +v\cdot \nabla_x) f_{z_0}]_{C^{2\alpha/3}_k(Q_{1/2})}\\
& \leq C \langle v_0\rangle^{-(q+2m)/3+7+7\alpha/3 + 6/\alpha + 2\alpha^2/9} \left( 1 + [\vv^m f]_{C^\alpha_k(\Omega(z_0))}^{3+2\alpha/3+3/\alpha}\right),
\end{split}
\end{equation}
Finally, we translate from $f_{z_0}$ to $f$, using the chain rule and \eqref{e:distance-compare}. In particular, with $\|S\|$ denoting the operator matrix norm of $S$, we have
\[
\begin{split}
[D_v^2 f]_{C^{2\alpha/3}_k(Q_{r_1/2}(z_0))} 
&\leq C r_1^{-2} \|S\|^{-2} (1+t_0^{-\alpha/3})\langle v_0\rangle^\alpha [D_v^2 f_{z_0}]_{C^{2\alpha/3}_k(Q_{1/2})}\\
& \leq C(1+t_0^{-1-\alpha/3})\langle v_0\rangle^{2+2\alpha/3}[D_v^2 f_{z_0}]_{C^{2\alpha/3}_k(Q_{1/2})},
\end{split}
\]
and
\[
\begin{split}
[(\partial_t + v\cdot\nabla_x)f]_{C^{2\alpha/3}_k(Q_{r_1/2}(z_0))}
&\leq C r_1^{-2}  (1+t_0^{-\alpha/3})\langle v_0\rangle^\alpha [(\partial_t + v\cdot \nabla_x) f_{z_0}]_{C^{2\alpha/3}_k(Q_{1/2})}\\
&\leq C(1+t_0^{-1-\alpha/3})\langle v_0\rangle^{2+\gamma+2\alpha/3}[(\partial_t + v\cdot \nabla_x) f_{z_0}]_{C^{2\alpha/3}_k(Q_{1/2})},
\end{split}
\]
which, combined with \eqref{e:intermediate-result}, imply the conclusion of the lemma.
\end{proof}

Finally, we are ready to prove the global $C^2$ estimate that is needed in our proof of the propagation of a H\"older modulus:

\begin{proof}[Proof of Proposition \ref{p:D2f-est}]
Let $z_0\in [\tau/2,\tau]\times\R^6$ be fixed, and let $r_1$ be defined by \eqref{e:r1} and $r_0 = \min\{1,\sqrt{t_0/2}\}$. As in Lemma \ref{l:schauder}, define $\Omega(z_0) = Q_{r_0}^{t,x}(z_0) \times \R^3_v$. Since our assumption on $f$ implies polynomial decay of all orders, we choose $m> 5 + \gamma + \alpha/3$ arbitrarily, and choose $q> m$ large enough that the exponent of $\langle v_0\rangle$ in Lemma \ref{l:schauder} is negative, i.e.
\[
-(q+2m)/3+9+3\alpha + 6/\alpha + 2\alpha^2/9 + \gamma \leq 0.
\]
We apply the weighted interpolation of Lemma \ref{l:exp-interp} (note that $\rho' = \rho_0$ with our choice of $\rho$), followed by Lemma \ref{l:schauder} with our choices of $m$ and $q$:
\[
\begin{split}
\|e^{\rho \vv^\beta} & D_v^2 f\|_{L^\infty(Q_{r_1/2}(z_0))} \\
&\leq 
C [D_v^2 f]_{C^{2\alpha/3}_k(Q_{r_1/2}(z_0))}^{1 - \frac {2\alpha} {6-\alpha}}  \|e^{\rho_0\vv^\beta} f\|_{C^\alpha_k(Q_{r_1/2}(z_0))}^{\frac{2\alpha}{6-\alpha}}\\
&\leq 
C \left[ \left( 1+t_0^{-1-\alpha/3}\right) (1 + [\vv^m f]_{C^\alpha_k(\Omega(z_0))}^{3+2\alpha/3+3/\alpha})\right]^{1- \frac{2\alpha}{6-\alpha}} \|e^{\rho_0\vv^\beta} f\|_{C^\alpha_k(Q_{r_1/2}(z_0))}^{\frac{2\alpha}{6-\alpha}}\\
&\leq 
C \left(1 + t_0^{-1 + \frac{\alpha^2}{6-\alpha}}\right)  \left(1+\|e^{\rho_0\vv^\beta}  f\|_{C^\alpha_k(\Omega(z_0))}^{(3+2\alpha/3+3/\alpha) (1 - \frac {2\alpha} {6-\alpha}) + \frac{2\alpha}{6-\alpha}}\right)
\end{split}
\]
where in the last line, we used $Q_{r_1/2}(z_0) \subset \Omega(z_0)$ and the crude upper bound $[\vv^m f]_{C^\alpha_k(\Omega(z_0))} \leq \|e^{\rho_0 \vv^\beta} f\|_{C^\alpha_k(\Omega(z_0))}$. Since $z_0 \in [\tau/2,\tau]\times\R^6$ was arbitrary, and $t_0\approx \tau$, the proof is complete. Note that the constant $C$ depends on $\|e^{\rho_0\vv^\beta} f\|_{L^\infty([0,\tau]\times\R^6)}$ due to the dependence on $K_0$ in Lemma \ref{l:schauder}.
\end{proof}

\end{document}